\documentclass[11pt]{amsart}

\pdfoutput=1 

\usepackage{esint,amssymb,geometry,verbatim,color}
\usepackage{graphicx}

\usepackage[colorlinks=true,urlcolor=blue,
citecolor=red,linkcolor=blue,linktocpage,pdfpagelabels,bookmarksnumbered,bookmarksopen]{hyperref}
\usepackage[hyperpageref]{backref}

\newtheorem{lm}{Lemma}[section]
\newtheorem{prop}[lm]{Proposition}
\newtheorem{teo}[lm]{Theorem}
\newtheorem{coro}[lm]{Corollary}
\theoremstyle{definition}
\newtheorem{oss}[lm]{Remark}
\newtheorem*{ack}{Acknowledgements}

\numberwithin{equation}{section}
\title[Regularity for anisotropic functionals]{Sobolev and Lipschitz regularity\\ for local minimizers\\ of widely degenerate anisotropic functionals}
\author[Brasco]{Lorenzo Brasco}
\author[Leone]{Chiara Leone}
\author[Pisante]{Giovanni Pisante}
\author[Verde]{Anna Verde}
\date{}


\address[L.\ Brasco]{Dipartimento di Matematica e Informatica
\newline\indent
Universit\`a degli Studi di Ferrara
\newline\indent
Via Machiavelli 35, 44121 Ferrara, Italy
\newline\indent
{\it and } Institut de Math\'ematiques de Marseille
\newline\indent
Aix-Marseille Universit\'e,
Marseille, France}
\email{lorenzo.brasco@unife.it}

\address[C.\ Leone \& A.\ Verde]{Dipartimento di Matematica ``R. Caccioppoli''
\newline\indent
Universit\`a degli Studi di Napoli ``Federico II''
\newline\indent
Via Cinthia, Complesso Universitario di Monte S. Angelo, 80126 Napoli, Italy}
\email{chiara.leone@unina.it}
\email{anna.verde@unina.it}

\address[G. Pisante]{Dipartimento di Matematica e Fisica
\newline\indent 
Seconda Universit\`a degli Studi di Napoli
\newline\indent 
Viale Lincoln 5, 81100 Caserta, Italy}
\email{giovanni.pisante@unina2.it}

\begin{document}

\begin{abstract}
We prove higher differentiability of bounded local minimizers to some widely degenerate functionals, verifying superquadratic anisotropic growth conditions. In the two dimensional case, we prove that local minimizers to a model functional are locally Lipschitz continuous functions, without any restriction on the anisotropy.
\end{abstract}

\maketitle

\begin{center}
\begin{minipage}{11cm}
\small
\tableofcontents
\end{minipage}
\end{center}

\section{Introduction}

\subsection{Overview}
In this paper we continue to investigate differentiability properties of local minimizers of convex functionals, exhibiting wide degeneracies and an {\it orthotropic} structure. The model case of functional we want to study is given by
\begin{equation}
\label{model}
\mathfrak{F}(u;\Omega')=\sum_{i=1}^N \frac{1}{p_i}\,\int_{\Omega'} (|u_{x_i}|-\delta_i)_+^{p_i}\,dx+\int_{\Omega'} f\,u\,dx,\qquad u\in W^{1,\mathbf{p}}_{loc}(\Omega),\ \Omega'\Subset\Omega,
\end{equation}
with $\delta_i\ge 0$ and $p_i\ge 2$. We denote $\mathbf{p}=(p_1,p_2,\dots,p_N)$ and 
\[
W^{1,\mathbf{p}}_{loc}(\Omega)=\Big\{u\in W^{1,1}_{loc}(\Omega)\, :\, u_{x_i}\in L^{p_i}_{loc}(\Omega),\ i=1,\dots,N\Big\}.
\]
The symbol $(\,\cdot\,)_+$ above stands for the positive part.
\par
For $p_1=\dots=p_N$, some results can be found in the recent papers \cite{BBJ} and \cite{BraCar}. We refer to the introduction of \cite{BraCar} for some motivations of this kind of functionals, arising from Optimal Transport problems with congestion effects.
\par
Our scope is to generalize these results to the {\it anisotropic} case, i.e. to the case where at least one of the exponents $p_i$ is different from the others. These functionals pertain to the class of {\it variational problems with non standard growth conditions}, first introduced by Marcellini in \cite{Ma91,Ma89}.
\par
Similar functionals have been considered in the past by the Russian school, see for example \cite{Ko} and \cite{UU}. More recently, they have been considered by many authors also in western countries. Among others, we mention (in alphabetical order) Bildhauer, Fuchs and Zhong \cite{BFZ2,BFZ} (where the terminology {\it splitting-type integrals} is used), Esposito, Leonetti and Mingione \cite{ELM}, Leonetti \cite{Leo}, Liskevich and Skrypnik \cite{LS} and Pagano \cite{Pa}. However, we point out that the type of degeneracy admitted in \eqref{model} is heavier than those of the above mentioned references, due to the presence of the coefficients $\delta_i\ge 0$ above.
\par
We observe that local minimizers of the functional \eqref{model} are local weak solution of the degenerate elliptic equation
\[
\sum_{i=1}^N \left((|u_{x_i}|-\delta_i)_+^{p_i-1}\,\frac{u_{x_i}}{|u_{x_i}|}\right)_{x_i}=f.
\]
The particular case $\delta_1=\dots=\delta_N=0$ and $p_1=\dots=p_N=p$ corresponds to
\[
\sum_{i=1}^N \left(|u_{x_i}|^{p-2}\,u_{x_i}\right)_{x_i}=f,
\]
which has been called {\it pseudo $p-$Laplace equation} in the recent literature. Here we prefer to use the terminology {\it orthotropic $p-$Laplace equation}, which seems more adapted and meaningful.
\par

\subsection{Main results}

Our first result is the Sobolev regularity for some nonlinear functions of the gradient of a {\it bounded} local minimizer. For $\mathbf{p}=(p_1,\dots,p_N)$, we will use the notation
\[
\mathbf{p}':=(p_1',\dots,p'_N),
\]
where $p_i'$ is the H\"older conjugate of $p_i$. 

\begin{teo}[Sobolev regularity for bounded minimizers]
\label{teo:sobolevpq}
Let $\ell\in\{1,\dots,N-1\}$ and $2\le p\le q$. We set 
\[
\mathbf{p}=(\underbrace{p,\cdots,p}_\ell,\underbrace{q,\cdots,q}_{N-\ell}),
\]
and let $u\in W^{1,{\bf p}}_{\rm loc}(\Omega)\cap L^\infty_{\rm loc}(\Omega)$ be a local minimizer of 
\[
\mathfrak{F}(u;\Omega')=\sum_{i=1}^N \int_{\Omega'} g_i(u_{x_i})\,dx+\int_{\Omega'} f\,u\,dx,
\] 
where $f\in W^{1,\mathbf{p}'}_{\rm loc}(\Omega)$ and $g_1,\dots,g_N:\mathbb{R}\to\mathbb{R}^+$  are $C^2$ convex functions such that
\[
\frac{1}{\mathcal{C}}\, (|s|-\delta)^{p-2}_+\le g_i''(s)\le \mathcal{C}\,(|s|^{p-2}+1),\qquad s\in\mathbb{R},\ i=1,\dots,\ell,
\]
\[
\frac{1}{\mathcal{C}}\, (|s|-\delta)^{q-2}_+\le g_i''(s)\le \mathcal{C}\,(|s|^{q-2}+1),\qquad s\in\mathbb{R},\ i=\ell+1,\dots,N,
\]
for some $\mathcal{C}\ge 1$ and $\delta\ge 0$. 
We set
\[
\mathcal{V}_i=V_i(u_{x_i}),\qquad \mbox{ where }\quad V_i(s)=\int_0^s \sqrt{g''_i(\tau)}\,d\tau,\qquad i=1,\dots,N.
\]
\begin{itemize}
\item If $\boxed{\ell= N-1}$ and $p,q$ satisfy 
\begin{equation}
\label{condtotN1}
\begin{split}
p\ge N-1&\qquad\hbox{ or } \qquad 
\left\{\begin{array}{rcl}
p&<&\dfrac{(N-2)^2}{N-1},\\
\\
q&<&\dfrac{(N-2)\,p}{(N-2)-p},\\
\end{array}
\right.\\
&\qquad\hbox{ or } \qquad 
\left\{\begin{array}{rcccl}
\dfrac{(N-2)^2}{N-1}&\le &p&<&N-1,\\
\\
&&q&<&\dfrac{p}{\Big(\sqrt{N-1}-\sqrt{p}\Big)^2},\\
\end{array}
\right.\\
\end{split}
\end{equation}
then we have $\mathcal{V}_i\in W^{1,2}_{\rm loc}(\Omega)$, for $i=1,\dots,N$.
\vskip.2cm
\item if $\boxed{1\le \ell\le N-2}$ and $p,q$ satisfy 
\begin{equation}
\label{condtot}
p\ge N-2,\qquad\hbox{ or } \qquad 
\left\{\begin{array}{rcl}
p&<&N-2,\\
\\
q&<&\dfrac{(N-2)\,p}{(N-2)-p},\\
\end{array}
\right.
\end{equation}
then we have $\mathcal{V}_i\in W^{1,2}_{\rm loc}(\Omega)$, for $i=1,\dots,N$.
\end{itemize}
Moreover, for every $B_r\Subset B_R\Subset \Omega$, we have 
\begin{equation}
\label{stimammerda}
\|\nabla \mathcal{V}_i\|_{L^2(B_r)}\le C=C\left(N,p,q,\mathcal{C},\delta,\mathrm{dist}(B_R,\partial\Omega),\|u\|_{L^\infty(B_R)},\|f\|_{W^{1,\mathbf{p}'}(B_R)}\right).
\end{equation}
\end{teo}
\begin{oss}
The previous result is proved under the additional assumption $u\in L^\infty_{\rm loc}(\Omega)$. Indeed, since the appearing of the celebrated counterexamples by Marcellini \cite{MaContro} and Giaquinta \cite{Gia}, it is well-known that local minimizers to this kind of functionals may be unbounded if $p$ and $q$ are too far apart (see also Hong's paper \cite{Ho}). Sharp conditions in order to get $u\in L^\infty_{\rm loc}$ can be found in \cite[Theorem 3.1]{FS}, see also the recent paper \cite{CMM}.
\end{oss}

\begin{oss}[Comparison with previous results, part I]  Theorem \ref{teo:sobolevpq} contains as a particular instance the scalar case of \cite[Theorem 2]{CDLL} by Canale, D'Ottavio, Leonetti and Longobardi, which still concerns bounded local minimizers. The latter corresponds to the particular case 
\[
\ell=N-1,\quad p=2\quad \mbox{ and }\quad \delta=0.
\]
However, even in this case, our result is stronger than \cite[Theorem 2]{CDLL}, since our conditions \eqref{condtotN1} are less restrictive for dimension $N\in\{2,3,4,5\}$.
In particular, we observe that for $N\in \{2,3\}$, the first condition in \eqref{condtotN1} is always fulfilled. Thus in low dimension we have Sobolev regularity no matter how large $q$ is, provided local minimizers are locally bounded.
\par
In the model case \eqref{model}, the result of Theorem \ref{teo:sobolevpq} boils down to
\[
\left(|u_{x_i}|-\delta_i\right)_+^\frac{p_i}{2}\,\frac{u_{x_i}}{|u_{x_i}|}\in W^{1,2}_{loc}(\Omega),\qquad i=1,\dots,N.
\]
In particular, for local weak solutions of the anisotropic orthotropic $p-$Laplace equation (i.e. for $\delta_i=0$), we get
\[
|u_{x_i}|^\frac{p_i-2}{2}\,u_{x_i}\in W^{1,2}_{loc}(\Omega),\qquad i=1,\dots,N.
\]
This is the analog of the well-known result $|\nabla u|^\frac{p-2}{2}\,\nabla u\in W^{1,2}_{loc}(\Omega)$,
for local weak solution of the $p-$Laplace equation 
\[
\Delta_p u:=\mathrm{div}(|\nabla u|^{p-2}\,\nabla u)=f,
\]
in the case $p\ge 2$ (see \cite[Lemma 3.1]{Uh}).
\end{oss}

We now restrict the discussion to the case of dimension $N=2$ and consider for simplicity the model case presented at the beginning. We can prove the Lipschitz regularity of local minimizers.
Namely, we obtain the following generalization of \cite[Theorem A]{BBJ}, the latter corresponding to the case $p_1=p_2$.
\begin{teo}[Lipschitz regularity in dimension $2$]
\label{teo:A}
Let $N=2$, $2\le p_1\le p_2$ and $\delta_1,\delta_2\ge 0$. Let $f\in W^{1,\mathbf{p}'}_{loc}(\Omega)$, then every local minimizer $u\in W^{1,\mathbf{p}}_{loc}(\Omega)$ of the functional
\[
\mathfrak{F}(u;\Omega')=\sum_{i=1}^2 \frac{1}{p_i}\,\int_{\Omega'} (|u_{x_i}|-\delta_i)_+^{p_i}\,dx+\int_{\Omega'} f\,u\,dx,\qquad u\in W^{1,\mathbf{p}}_{loc}(\Omega),\ \Omega'\Subset\Omega,
\]
is a locally Lipschitz continuous function.
\end{teo}
\begin{oss}[Comparison with previous results, part II] 
To the best of our knowledge, this result is new already in the simpler case of the functional
\[
u\mapsto\sum_{i=1}^2 \frac{1}{p_i}\,\int_{\Omega'} |u_{x_i}|^{p_i}\,dx.
\]
The only result of this type we are aware of is the pioneering one \cite[Theorem 1]{UU} by Ural'tseva and Urdaletova. Though their result holds for every dimension $N\ge 2$, this needs the additional assumptions 
\[
p_1\ge 4\qquad\mbox{ and }\qquad p_N<2\,p_1.
\]
On the contrary, these restrictions are not needed in Theorem \ref{teo:A}.
\end{oss}

\subsection{Some comments on the proofs}
Let us spend some words on the proofs of our main results. As for Theorem \ref{teo:A}, the proof is the same as that of \cite[Theorem A]{BBJ}, up to some technical modifications. This is based on a trick introduced in \cite{BBJ}: this permits to obtain Caccioppoli inequalities for convex functions of the gradient $\nabla u$, by combining the linearized equation and the Sobolev regularity of Theorem \ref{teo:sobolevpq}. One can then build an iterative scheme of reverse H\"older inequalities and obtain the desired result by performing a Moser's iteration.
The trick is a two-dimensional one and does not seem possible to extend it to higher dimensional cases. On the other hand, we show here that the limitation $p_1=p_2$ is not needed and the same proof works for $p_1<p_2$ as well.
\par
On the contrary, the proof of Theorem \ref{teo:sobolevpq} contains a crucial novelty, which permits to improve the range of validity of Sobolev regularity, compared to previous results based on similar proofs. In order to neatly explain this point, we briefly resume the strategy for proving Theorem \ref{teo:sobolevpq} in the model case 
\[
\mathfrak{F}(u;\Omega')=\sum_{i=1}^N \frac{1}{p_i}\,\int_{\Omega'} (|u_{x_i}|-\delta_i)_+^{p_i}\,dx,
\]
with the exponents $p_1\le p_2\le \dots p_N$ which could all differ.
The starting point of the proof is differentiating the relevant Euler-Lagrange equation in a discrete sense, i.e. we use the {\it Nirenberg's method of incremental ratios}. This is very classical and permits to estimate integrated finite differences of the type
\begin{equation}
\label{quoziente}
\int_{B_r} \left|\frac{\mathcal{V}_i(\cdot+h\,\mathbf{e}_j)-\mathcal{V}_i}{|h|^t}\right|^2\,dx,\qquad  i,j=1,\dots,N,
\end{equation} 
by appealing to the monotonicity properties of the operator. The integrals \eqref{quoziente} are estimated in terms of quantities like
\begin{equation}
\label{sergio!}
\int_{B_R} \,\left|\frac{u(\cdot+h\,\mathbf{e}_j)-u}{|h|^{t}}\right|^{p_i}\,dx.
\end{equation}
This is a (possibly fractional) derivative of $u$ in the $j-$th direction, but raised to the power competing to the $i-$th direction. In the case $p_1=\dots=p_N=p\ge 2$, one can take $t=1$ and conclude directly that $\mathcal{V}_i\in W^{1,2}_{loc}$, thanks to the finite differences characterization of Sobolev spaces.
\par 
On the contrary, the anisotropic case is subtler. We first observe that since by assumptions $u_{x_j}\in L^{p_j}_{loc}$, when $i\le j$ we could take again $t=1$ in \eqref{quoziente} and \eqref{sergio!} and obtain full differentiability in these directions. For example, this is always the case if $j=N$, i.e. if we derive the equation in the $x_N$ direction, corresponding to the direction of maximal growth of the functional.
\par
On the other hand, when $j\le N-1$ we have to pay attention to the ``bad directions'', corresponding to terms \eqref{sergio!} with $i\ge j+1$. Indeed, in this case {\it we do not know that} $u_{x_j}\in L^{p_i}_{loc}$. Rather, we choose $0<t<1$ (depending on the ratio $p_j/p_i$) and we use a $L^\infty$--$W^{1,p_j}$ interpolation in order to control this term (it is here that the assumption $u\in L^\infty_{loc}$ comes into play). By proceeding in this way, we get for every $i=1,\dots,N$ 
\begin{equation}
\label{primostep}
\mathcal{V}_i\ \mbox{ is weakly differentiable of order }\ \frac{p_j}{p_N}\ \mbox{ in the direction }\ \mathbf{e}_j.
\end{equation}
However, this is not the end of the story. Indeed, this information now entails that $\mathcal{V}_i$ (and thus $u_{x_i}$) enjoys better integrability properties, by fractional Sobolev embeddings. This in turn implies that we can re-initialize the previous scheme and exploit this new integrability in order to have a better control on \eqref{sergio!}. As a consequence, we can improve \eqref{primostep}. The final outcome is thus obtained by a (finite) iteration of the scheme just described.
\par 
Up to now, the proof is very similar to that of \cite{CDLL}.
The main difference is in the way we exploit \eqref{primostep} in order to improve the integrability of $\mathcal{V}_i$. In a nutshell, what usually one does is to extrapolate from \eqref{primostep} the weaker {\it isotropic} information
\[
\mathcal{V}_i\in W^{\frac{p_1}{p_N},2}_{\rm loc}(\Omega),
\]
and then use the Sobolev embedding for usual fractional Sobolev-Slobodecki\u{\i} spaces. Then the algorithm runs as described above. Since in every direction we pass from $p_j/p_N$ to $p_1/p_N$, each time $p_j>p_1$  this gives rise to a loss of information which may be important. 
\par
In this paper, on the contrary, we take advantage of the full information contained in \eqref{primostep}. The latter means that each $\mathcal{V}_i$ is contained in an {\it anisotropic Besov-Nikol'ski\u{\i} space}, where the anisotropy is now in the order of differentiability (we refer to Section \ref{sec:2bn} for the relevant definition). As one may expect, such a space 
has an improved Sobolev embedding, thus by proceeding in this alternative way the gain of integrability is strictly better at each step.
\par
This kind of anisotropic spaces and their embeddings seem to be completely overlooked or neglected by the recent literature on anisotropic problems, we refer to Nikol'ski\u{\i}'s monography \cite{Ni} for a comprehensive treatment of the subject (an alternative approach can be also find in Triebel's book \cite{Tr}). We believe on the contrary this to be the natural setting for the problem and the natural tools to be exploited. These spaces are also briefly treated in the classical monography \cite{KJF} by Kufner, John and Fu\v{c}\'ik (see Sections 2 and 4 of \cite[Chapter 8]{KJF}). 
\begin{oss}[Why two exponents only?]
After the previous description of the method of proof for the general case of $p_1\le p_2\le \dots\le p_N$, the reader may be perplexed to see that in Theorem \ref{teo:sobolevpq} we confine ourselves to the case of only two different exponents $p\le q$. The reason is easy to explain: the iterative scheme described above quickly becomes fairly intricate, in the general case $p_1\le p_2\le \dots\le p_N$. In particular, when one tries to perform the iteration, at each step many subcases should be discussed by making the proof very difficult to be written (and read).
For this reason, we preferred to confine our discussion to the case of two exponents. 
\par 
At the same time, we believe our approach to be interesting and promising. Thus we explicitely write down the iterative step in the general case of $N$ exponents, without running the scheme up to the end, see Propositions \ref{lm:step1} and \ref{lm:stepk} below. These are valid under the assumption $u\in L^\infty_{\rm loc}(\Omega)$ and without restrictions on the spreadness of the exponents, thus they can be used in the general case $p_1\le p_2\le \dots\le p_N$ to obtain partial higher differentiability results.
\end{oss}

\subsection{Plan of the paper}
In Section \ref{sec:2} we fix the notations and we set the preliminaries results needed throughout the paper, particularly focusing on embedding theorems for anisotropic Besov-Nikol'ski\u{\i} spaces. In Section \ref{sec:3} we present, in a general form, the details of the scheme for improving differentiability roughly described above. Sections \ref{sec:4} and \ref{sec:5} are devoted to the proofs of Theorem \ref{teo:sobolevpq} and Theorem \ref{teo:A} respectively. Some useful technical inequalities are finally collected in the Appendix.    

\begin{ack}
The authors wish to thank Nicola Fusco for bringing reference \cite{FS} to their attention.
The authors are members of the Gruppo Nazionale per l'Analisi Matematica, la Probabilit\`a
e le loro Applicazioni (GNAMPA) of the Istituto Nazionale di Alta Matematica (INdAM).
Part of this work has been done during some visits of the first author to Napoli, as well as during the ``XXVI Italian Workshop on Calculus of Variations'' held in Levico Terme, in January 2016. Hosting institutions and organizers are gratefully acknowledged.  
\end{ack}

\section{Preliminaries}
\label{sec:2}

\subsection{Notation}

Given $\mathbf{h}\in\mathbb{R}^N\setminus\{0\}$, for a measurable function $\psi:\mathbb{R}^N\to\mathbb{R}$ we introduce the notation
\[
\psi_\mathbf{h}(x):=\psi(x+\mathbf{h})\qquad \mbox{ and }\qquad \delta_\mathbf{h} \psi(x):=\psi_\mathbf{h}(x)-\psi(x).
\]
We recall that for every pair of functions $\varphi,\psi$ we have
\begin{equation}
\label{leibniz}
\delta_\mathbf{h} (\varphi\,\psi)=(\delta_\mathbf{h} \varphi)\,\psi+\varphi_\mathbf{h}\,(\delta_\mathbf{h} \psi).
\end{equation}
We also use the notation
\[
\delta^2_\mathbf{h} \psi(x):=\delta_\mathbf{h}(\delta_\mathbf{h} \psi(x))=\psi(x+2\,\mathbf{h})+\psi(x)-2\,\psi(x+\mathbf{h}).
\]
We indicate by $\{\mathbf{e}_1,\dots,\mathbf{e}_N\}$ the canonical basis of $\mathbb{R}^N$.
\vskip.2cm\noindent
Given $1\le p_1\le p_2\leq \dots\leq p_N$, we denote $\mathbf{p}=(p_1,p_2,\dots,p_N)$. Let $E\subset \mathbb{R}^N$ be an open set, we define the anisotropic Sobolev spaces
\[
W^{1,\mathbf{p}}(E)=\{u\in W^{1,1}(E)\, :\, u_{x_i}\in L^{p_i}(E),\ i=1,\dots,N\},
\]
and
\[
W^{1,\mathbf{p}}_0(E)=\{u\in W^{1,1}_0(E)\, :\, u_{x_i}\in L^{p_i}(E),\ i=1,\dots,N\}.
\]
We define the harmonic mean $\overline p$ of the exponents $p_1\le p_2\le\dots\le p_N$
\[
\frac{1}{\overline{p}}=\frac{1}{N}\, \sum_{i=1}^N \frac{1}{p_i},
\]
then the associated Sobolev-type exponent is defined by
\[
\overline{p}^*=\left\{\begin{array}{ll}
\dfrac{N\, \overline{p}}{N-\overline{p}},& \mbox{ if } 1\le \overline p<N,\\
+\infty, & \mbox{ if } \overline p>N.
\end{array}
\right.
\] 
Finally, for $0<t<1$ and $1\le p<\infty$ we denote by $W^{t,p}(\mathbb{R}^N)$ the {\it Sobolev-Slobodecki\u{\i} space}, i.\,e.
\[
W^{t,p}(\mathbb{R}^N)=\Big\{u\in L^p(\mathbb{R}^N)\, :\, [u]_{W^{t,p}(\mathbb{R}^N)}<+\infty\Big\},
\]
where 
\[
[u]_{W^{t,p}(\mathbb{R}^N)}^p=\int_{\mathbb{R}^N} \int_{\mathbb{R}^N} \frac{|u(x)-u(y)|^p}{|x-y|^{N+t\,p}}\,dx\,dy.
\]
Though we will not need this, we recall that $W^{t,p}(\mathbb{R})$ can be seen a particular instance of the larger class of Besov spaces.

\subsection{Embedding for anisotropic Sobolev spaces}
 We collect here a couple of embedding results that will be needed in the sequel. The first one is well-known\footnote{This result is usually attributed to Troisi in the literature of western countries, see \cite{Tro}. However, Trudinger in \cite{Tru} attributes the result for $\overline{p}\not=N$ to Nikol'ski\u{\i}, whose paper \cite{Ni2} appeared before \cite{Tro}. In any case, the methods of proof are different.}, a proof can be found for example in \cite[Theorem 1 \& Corollary 1]{Tru}.
\begin{teo}[Anisotropic Sobolev embeddings]
\label{teo:ST}
Let $\Omega\subset\mathbb{R}^N$ be an open set, then for every $u\in W^{1,\mathbf{p}}_0(\Omega)$ we have:
\begin{enumerate}
\item if $\overline p<N$
\[
c\, \|u\|_{L^{\overline p^*}(\Omega)}\le \sum_{i=1}^N \left(\int_{\Omega} |u_{x_i}|^{p_i}\,dx\right)^\frac{1}{p_i},
\]
for a constant $c=c(N,\mathbf{p})>0$;
\vskip.2cm
\item if $\overline p=N$ and $|\Omega|<+\infty$, for every $1\le \chi<\infty$
\[
c\, \|u\|_{L^{\chi}(\Omega)}\le |\Omega|^\frac{1}{\chi}\,\sum_{i=1}^N \left(\int_{\Omega} |u_{x_i}|^{p_i}\,dx\right)^\frac{1}{p_i}.
\]
for a constant $c=c(N,\mathbf{p},\chi)>0$;
\vskip.2cm
\item if $\overline p>N$ and $|\Omega|<+\infty$
\[
c\, \|u\|_{L^{\infty}(\Omega)}\le |\Omega|^{\frac{1}{N}-\frac{1}{\overline p}}\,\sum_{i=1}^N \left(\int_{\Omega} |u_{x_i}|^{p_i}\,dx\right)^\frac{1}{p_i},
\]
for a constant $c=c(N,\mathbf{p})>0$.
\end{enumerate}
\end{teo}
The next embedding result is stated in \cite[Theorem 1]{KK}. We provide a proof for the reader's convenience.
\begin{prop}
\label{prop:KK}
Let $\Omega\subset\mathbb{R}^N$ be an open set and let $\mathbf{p}=(p_1,\dots,p_N)$ be such that 
\begin{equation}
\label{epoi?!}
1< p_1\le \dots\le p_N<\overline p^*.
\end{equation}
Then for every $E\Subset\Omega$ we have
\[
W^{1,\mathbf{p}}(\Omega)\hookrightarrow L^{\overline p^*}(E),\qquad \mbox{ if } \overline p\not=N,
\]
and
\[
W^{1,\mathbf{p}}(\Omega)\hookrightarrow L^{\chi}(E),\ \mbox{ for every } 1\le \chi<\infty,\qquad \mbox{ if } \overline p=N.
\]
\end{prop}
\begin{proof}
Let us fix two concentric balls $B_{\varrho_0}\Subset B_{R_0}\Subset \Omega$. For every $\varrho_0\le \varrho<R\le R_0$ we take a standard cut-off function $\eta\in C^\infty_0(B_{R})$ such that $\eta\equiv 1$ on $B_\varrho$, with
\[
\|\nabla \eta\|_{L^\infty(\mathbb{R}^N)}\le \frac{C}{R-\varrho},
\]
for some universal constant $C>0$. Let $u\in W^{1,\mathbf{p}}(\Omega)$, then for every $M>0$ we define
\[
u_M=\min\Big\{|u|,\,M\Big\}\in W^{1,\mathbf{p}}(\Omega)\cap L^\infty(\Omega),
\]
and finally take $u_M\,\eta\in W^{1,\mathbf{p}}_0(B_R)\subset W^{1,\mathbf{p}}_0(\Omega)$. Let us suppose for simplicity that $\overline p<N$, by Theorem \ref{teo:ST} we have
\[
\sum_{i=1}^N \left(\int_\Omega \left|(u_M\,\eta)_{x_i}\right|^{p_i}\,dx\right)^\frac{1}{p_i}\ge c\, \left(\int_\Omega |u_M\,\eta|^{\overline p^*}\,dx\right)^\frac{1}{\overline p^*},
\]
for some $c=c(N,\mathbf{p})>0$.
By using the properties of $\eta$, with simple manipulations we get
\begin{equation}
\label{KKquasi}
\sum_{i=1}^N \left(\int_\Omega \left|(u_M)_{x_i}\right|^{p_i}\,dx\right)^\frac{1}{p_i}+\sum_{i=1}^N\,\frac{C}{R-\varrho}\, \left(\int_{B_R} \left|u_M\right|^{p_i}\,dx\right)^\frac{1}{p_i}\ge c\, \left(\int_{B_\varrho} |u_M|^{\overline p^*}\,dx\right)^\frac{1}{\overline p^*},
\end{equation}
for a possibly different constant $c>0$, still depending on $N$ and $\mathbf{p}$ only. We now observe that by hypothesis \eqref{epoi?!} we have $1<p_i<\overline p^*$, thus by interpolation in Lebesgue spaces
\[
\begin{split}
\left(\int_{B_R} |u_M|^{p_i}\,dx\right)^\frac{1}{p_i}&\le \left(\int_{B_R} |u_M|\,dx\right)^{(1-\vartheta_i)}\,\left(\int_{B_R} |u_M|^{\overline p^*}\,dx\right)^{\frac{\vartheta_i}{\overline p^*}},
\end{split}
\]
where
\[
\vartheta_i=\frac{p_i-1}{p_i}\,\frac{\overline p^*}{\overline p^*-1}\in(0,1).
\]
Thus by Young inequality we get
\[
\begin{split}
\frac{1}{R-\varrho}\,\left(\int_{B_R} |u_M|^{p_i}\,dx\right)^\frac{N}{p_i}&\le \left[\frac{\tau^{-\vartheta_i}}{R-\varrho}\,\left(\int_{B_R} |u_M|\,dx\right)^{(1-\vartheta_i)}\right]^\frac{1}{1-\vartheta_i}\\
&+\tau\,\left(\int_{B_R} |u_M|^{\overline p^*}\,dx\right)^\frac{1}{\overline p^*},
\end{split}
\]
for every $0<\tau<1$. By choosing $\tau$ small enough and using the previous estimate in \eqref{KKquasi}, we get
\[
\begin{split}
\sum_{i=1}^N \left(\int_\Omega \left|(u_M)_{x_i}\right|^{p_i}\,dx\right)^\frac{1}{p_i}&+\sum_{i=1}^N\frac{C}{(R-\varrho)^\frac{1}{1-\vartheta_i}}\,\left(\int_\Omega |u_M|\,dx\right)\\
&+\frac{c}{2}\, \left(\int_{B_R} \left|u_M\right|^{\overline p^*}\,dx\right)^\frac{1}{\overline p^*}\ge c\, \left(\int_{B_\varrho} |u_M|^{\overline p^*}\,dx\right)^\frac{1}{\overline p^*}.
\end{split}
\]
The previous holds for every $\varrho_0\le \varrho <R\le R_0$, from \cite[Lemma 6.1]{Gi} we obtain
\[
C\,\sum_{i=1}^N \left(\int_\Omega \left|(u_M)_{x_i}\right|^{p_i}\,dx\right)^\frac{1}{p_i}+\sum_{i=1}^N\frac{C}{(R_0-\varrho_0)^\frac{1}{1-\vartheta_i}}\,\int_\Omega |u_M|\,dx\ge \left(\int_{B_{\varrho_0}} |u_M|^{\overline p^*}\,dx\right)^\frac{1}{\overline p^*},
\]
for some constant $C=C(N,\mathbf{p})>0$. By arbitrariness of $B_{r_0}\Subset B_{R_0}\Subset\Omega$, for every $E\Subset \Omega$ a standard covering argument leads to
\begin{equation}
\label{KK}
C\,\left[\sum_{i=1}^N \left(\int_\Omega \left|u_{x_i}\right|^{p_i}\,dx\right)^\frac{1}{p_i}+\int_\Omega |u|\,dx\right]\ge \left(\int_{E} |u_M|^{\overline p^*}\,dx\right)^\frac{1}{\overline p^*},
\end{equation}
for some constant $C=C(N,\mathbf{p},\mathrm{dist}(E,\partial\Omega))>0$. In the previous inequality we also used that
\[
\left|u_{x_i}\right|\ge \left|(u_M)_{x_i}\right|\qquad \mbox{ and }\qquad |u_M|\le |u|,\qquad \mbox{ almost eveywhere on }\Omega.
\]
If we now take the limit as $M$ goes to $+\infty$ in \eqref{KK}, we get the desired result.
\end{proof}
\begin{oss}[Optimality of assumptions]
In general we can not take $E=\Omega$ or $p_N\ge \overline p^*$ in the previous result, see \cite{KK} for a counter-example. On the contrary, the hypothesis $p_1>1$ can be easily removed and we can relax it to $p_1\ge 1$. We leave the verification of this fact to the reader.
\end{oss}

\subsection{Anisotropic Besov-Nikol'ski\u{\i} spaces}
\label{sec:2bn}

Let $\psi\in L^p(\mathbb{R}^N)$, for $p\ge 1$ and $0<t\le 1$ we define the quantities
\begin{equation}
\label{h0}
[\psi]_{\mathfrak{n}^{t,p}_{\infty,i}}=\sup_{|h|>0} \left\|\frac{\delta_{h\mathbf{e}_i} \psi}{|h|^t}\right\|_{L^p(\mathbb{R}^N)},\qquad i=1,\dots,N,
\end{equation}
and
\begin{equation}
\label{h1}
[\psi]_{\mathfrak{b}^{t,p}_{\infty,i}}=\sup_{|h|>0} \left\|\frac{\delta^2_{h\mathbf{e}_i} \psi}{|h|^t}\right\|_{L^p(\mathbb{R}^N)},\qquad i=1,\dots,N.
\end{equation}
\begin{lm}
\label{lm:ugualimadiversi}
Let $0<t<1$, then for every $\psi\in L^p(\mathbb{R}^N)$ we have
\begin{equation}
\label{paragone}
\frac{1}{2}\,[\psi]_{\mathfrak{b}^{t,p}_{\infty,i}}\le [\psi]_{\mathfrak{n}^{t,p}_{\infty,i}}\le \frac{C}{1-t}\,\left[[\psi]_{\mathfrak{b}^{t,p}_{\infty,i}}+\|\psi\|_{L^p(\mathbb{R}^N)}\right].
\end{equation}
For $t=1$, for every $\psi\in L^p(\mathbb{R}^N)$ we have
\[
\frac{1}{2}\,[\psi]_{\mathfrak{b}^{t,p}_{\infty,i}}\le [\psi]_{\mathfrak{n}^{t,p}_{\infty,i}},
\]
and there exists $\psi_0\in L^p(\mathbb{R}^N)$ such that
\[
[\psi_0]_{\mathfrak{b}^{1,p}_{\infty,i}}<+\infty \qquad \mbox{ and }\qquad [\psi_0]_{\mathfrak{n}^{1,p}_{\infty,i}}=+\infty.
\]
\end{lm}
\begin{proof}
The first inequality in \eqref{paragone} is a plain consequence of triangle inequality and invariance by translations of $L^p$ norms. The second one can be proved by using a standard device, see \cite[Chapter 2.6]{Tri}.
\par
For $t=1$, an instance of function with the properties above can be found in \cite[Example page 148]{St}.
\end{proof}
If $\mathbf{t}=(t_1,\dots,t_N)\in (0,1]^N$, by following Nikol'ski\u{\i} we define the corresponding {\it anisotropic Besov-Nikol'ski\u{\i}} spaces as
\[
\mathcal{N}^{\mathbf{t},p}_{\infty}(\mathbb{R}^N):=\left\{\psi\in L^p(\mathbb{R}^N)\, :\, \sum_{i=1}^N [\psi]_{\mathfrak{n}^{t_i,p}_{\infty,i}}<+\infty\right\},
\]
and\footnote{In \cite{Ni} this space is denoted by $H^{\mathbf{t}}_p$ and is seen to be a particular instance of a general class of anisotropic Besov spaces noted $B^{\mathbf{t}}_{p\theta}$, with $1\le \theta\le \infty$.}
\[
\mathcal{B}^{\mathbf{t},p}_{\infty}(\mathbb{R}^N):=\left\{\psi\in L^p(\mathbb{R}^N)\, :\, \sum_{i=1}^N [\psi]_{\mathfrak{b}^{t_i,p}_{\infty,i}}<+\infty\right\},
\]
see \cite[pages 159--161]{Ni}. We equip them with the norms
\[
\|\psi\|_{\mathcal{N}^{\mathbf{t},p}_\infty(\mathbb{R}^N)}:=\|\psi\|_{L^p(\mathbb{R}^N)}+\sum_{i=1}^N [\psi]_{\mathfrak{n}^{t_i,p}_{\infty,i}}\qquad \mbox{ and }\qquad \|\psi\|_{\mathcal{B}^{\mathbf{t},p}_\infty(\mathbb{R}^N)}:=\|\psi\|_{L^p(\mathbb{R}^N)}+\sum_{i=1}^N [\psi]_{\mathfrak{b}^{t_i,p}_{\infty,i}}.
\]
From now on we will always implicitly assume that $t_1 \leq t_2 \leq \dots \leq t_N$. Before going on, a couple of comments are in order.
\begin{oss}[Comparison of the two spaces]
\label{oss:comparison}
By Lemma \ref{lm:ugualimadiversi} we get that if $0<t_1\le \dots\le t_N<1$, then
\[
\mathcal{N}^{\mathbf{t},p}_\infty(\mathbb{R}^N)=\mathcal{B}^{\mathbf{t},p}_\infty(\mathbb{R}^N).
\]
On the contrary, if $t_i=1$ for some $i\in \{1,\dots,N\}$, then
\[
\mathcal{N}^{\mathbf{t},p}_\infty(\mathbb{R}^N)\hookrightarrow \mathcal{B}^{\mathbf{t},p}_\infty(\mathbb{R}^N)\qquad \mbox{ and }\qquad \mathcal{N}^{\mathbf{t},p}_\infty(\mathbb{R}^N)\not=\mathcal{B}^{\mathbf{t},p}_\infty(\mathbb{R}^N).
\]
Moreover, we recall that if
\[
[\psi]_{\mathfrak{n}^{1,p}_{\infty,i}}<+\infty,\qquad \mbox{ for some } i\in\{1,\dots,N\},
\] 
then its distributional derivative $\psi_{x_i}$ belongs to $L^p(\mathbb{R}^N)$, see \cite[Theorem 4.8]{Ni}. 
\end{oss}
\begin{oss}
In the isotropic case $t_1=\dots=t_N=t$ with $0<t<1$, we simply denote these spaces by $\mathcal{N}^{t,p}_\infty(\mathbb{R}^N)$ and $\mathcal{B}^{t,p}_\infty(\mathbb{R}^N)$. By Lemma \ref{lm:ugualimadiversi} the seminorms
\[
\psi\mapsto\sum_{i=1}^N [\psi]_{\mathfrak{n}^{t,p}_{\infty,i}}\qquad \mbox{ and }\qquad \psi\mapsto\sum_{i=1}^N [\psi]_{\mathfrak{b}^{t,p}_{\infty,i}},
\]
are equivalent. Moreover, these in turn are equivalent to
\[
\psi\mapsto\sup_{|\mathbf{h}|>0} \left\|\frac{\delta_{\mathbf{h}} \psi}{|\mathbf{h}|^t}\right\|_{L^p(\mathbb{R}^N)}\qquad \mbox{ or }\qquad \psi\mapsto\sup_{|\mathbf{h}|>0} \left\|\frac{\delta^2_{\mathbf{h}} \psi}{|\mathbf{h}|^t}\right\|_{L^p(\mathbb{R}^N)}.
\]
\end{oss}
The next very simple result asserts that $\mathcal{N}^{\mathbf{t},p}_\infty(\mathbb{R}^N)$ and $\mathcal{B}^{\mathbf{t},p}_\infty(\mathbb{R}^N)$ do not change, if in \eqref{h0} and \eqref{h1} the supremum is restricted to $0<|h|<h_0$. The easy proof is left to the reader.
\begin{lm}
\label{lm:facile}
Let $0<t\le 1$ and $\psi\in L^p(\mathbb{R}^N)$, then for every $h_0>0$ and every $i=1,\dots,N$ we have
\[
[\psi]_{\mathfrak{n}^{t,p}_{\infty,i}}\le \sup_{0<|h|<h_0} \left\|\frac{\delta_{h\mathbf{e}_i} \psi}{|h|^{t}}\right\|_{L^p(\mathbb{R}^N)}+2\,h_0^{-t}\,\|\psi\|_{L^p(\mathbb{R}^N)},
\]
and
\[
[\psi]_{\mathfrak{b}^{t,p}_{\infty,i}}\le \sup_{0<|h|<h_0} \left\|\frac{\delta^2_{h\mathbf{e}_i} \psi}{|h|^{t}}\right\|_{L^p(\mathbb{R}^N)}+3\,h_0^{-t}\,\|\psi\|_{L^p(\mathbb{R}^N)}.
\]
\end{lm}
The following interpolation-type result as well is straightforward.
\begin{lm}
\label{lm:embeddingbanale}
Let $0<t <s\le 1$  and $\psi\in L^p(\mathbb{R}^N)$, then for every $i=1,\dots,N$ we have
\begin{equation}
\label{abbassa}
[\psi]_{\mathfrak{b}^{t,p}_{\infty,i}}\le s\,t^{-\frac{t}{s}}\,\left(\frac{3}{s-t}\right)^\frac{s-t}{s}\,[\psi]_{\mathfrak{b}^{s,p}_{\infty,i}}^\frac{t}{s}\,\|\psi\|_{L^p(\mathbb{R}^N)}^\frac{s-t}{s}.
\end{equation}
In particular, we have the continuous embedding $\mathcal{B}^{\mathbf{t},p}_\infty(\mathbb{R}^N)\hookrightarrow \mathcal{B}^{t_1,p}_\infty(\mathbb{R}^N)$. 
\end{lm}
\begin{proof}
We can suppose that the right-hand side of \eqref{abbassa} is finite, otherwise there is nothing to prove. For every $i=1,\dots,N$ and $0<|h|<h_0$, we have
\[
\left\|\frac{\delta^2_{h\mathbf{e}_i} \psi}{|h|^{t}}\right\|_{L^p(\mathbb{R}^N)}\le h_0^{s-t}\,\left\|\frac{\delta^2_{h\mathbf{e}_i} \psi}{|h|^{s}}\right\|_{L^p(\mathbb{R}^N)}.
\]
By taking the supremum over $0<|h|<h_0$ and using Lemma \ref{lm:facile} , we obtain
\[
[\psi]_{\mathfrak{b}^{t,p}_{\infty,i}}\le h_0^{s-t}\,[\psi]_{\mathfrak{b}^{s,p}_{\infty,i}}+3\,h_0^{-t}\,\|\psi\|_{L^p(\mathbb{R}^N)}.
\]
If we now optimize in $h_0$, we get the claimed inequality.
\end{proof}
We need the following embedding property in standard Sobolev-Slobodecki\u{\i} spaces.
\begin{lm}
\label{lm:scendi!}
Let $\mathbf{t}=(t_1,\dots,t_N)\in (0,1]^N$. 
Then we have the continuous embeddings
\[
\mathcal{N}^{\mathbf{t},p}_\infty(\mathbb{R}^N)\hookrightarrow\mathcal{B}^{\mathbf{t},p}_\infty(\mathbb{R}^N)\hookrightarrow W^{\kappa,p}(\mathbb{R}^N),\qquad \mbox{ for every }0<\kappa<t_1.
\]
\end{lm}
\begin{proof}
The first embedding follows from Remark \ref{oss:comparison}. Then it is sufficient to combine Lemma \ref{lm:embeddingbanale} with the well-known embedding $\mathcal{B}^{t_1,p}_\infty(\mathbb{R}^N)\hookrightarrow W^{\kappa,p}(\mathbb{R}^N)$, valid for every $0<\kappa<t_1$ (see \cite[Section 8.2.5]{KJF}).
\end{proof}
Finally, the following embedding result in Lebesgue spaces will be important.
\begin{teo}
\label{teo:nikolskiembedding}
Let $1\le p\le N$ and let $\mathbf{t}=(t_1,\dots,t_N)\in(0,1]^N$ be such that $t_1<1$. If we set 
\[
\gamma:=\sum_{i=1}^N \frac{1}{t_i},
\] 
then we have the continuous embeddings
\[
\mathcal{N}^{\mathbf{t},p}_\infty(\mathbb{R}^N)\hookrightarrow\mathcal{B}^{\mathbf{t},p}_{\infty}(\mathbb{R}^N)\hookrightarrow L^{p\,\chi}(\mathbb{R}^N),\qquad \mbox{ for every }\ 1\leq \chi<\frac{\gamma}{\gamma-p}.
\]
\end{teo}
\begin{proof}
It is sufficient to prove the embedding for $\mathcal{B}^{\mathbf{t},p}_{\infty}(\mathbb{R}^N)$.
We first observe that $\gamma> N\ge p$, thus the condition on $\chi$ is well-posed.
By \cite[Chapter 6, Section 3]{Ni} we have the embedding
\[
\mathcal{B}^{\mathbf{t},p}_{\infty}(\mathbb{R}^N)\hookrightarrow \mathcal{B}^{\mathbf{s},q}_{\infty}(\mathbb{R}^N),
\]
where $\mathbf{s}=(s_1,\dots,s_N)$ and $q>p$ are such that
\[
s_i=\beta\,t_i\qquad \mbox{ and }\qquad \beta=1-\left(\frac{1}{p}-\frac{1}{q}\right)\,\gamma>0.
\]
By Lemma \ref{lm:scendi!} and Sobolev inequality for Sobolev-Slobodecki\u{\i} spaces (see for example \cite[Theorem 1.73]{Tr}), for every $0<\kappa<s_1=\beta\,t_1$ we have
\[
\mathcal{B}^{\mathbf{s},q}_{\infty}(\mathbb{R}^N)\hookrightarrow W^{\kappa,q}(\mathbb{R}^N)\hookrightarrow L^\frac{N\,q}{N-\kappa\,q}(\mathbb{R}^N).
\]
We now observe that we can take $\beta>0$ arbitrarily close to $0$. Since we have
\[
q=\frac{p\,\gamma}{\gamma+p\,\beta-p}\qquad \mbox{ and }\qquad \frac{N\,q}{N-\kappa\,q}=\frac{N}{N-\kappa\,\displaystyle\frac{p\,\gamma}{\gamma+p\,\beta-p}}\,\frac{p\,\gamma}{\gamma+p\,\beta-p}
\] 
this implies that the last exponent can be taken as close as desired to $p\,\gamma/(\gamma-p)$ (observe that $\kappa$ converges to $0$ as $\beta$ goes to $0$).
\end{proof}
\begin{oss}
We observe that for the isotropic case $t_1=\dots=t_N=t\in(0,1]$ the exponent $p\,\gamma/(\gamma-p)$ coincide with the usual Sobolev exponent $N\,p/(N-t\,p)$ for the space $W^{t,p}(\mathbb{R}^N)$ in the case $t\,p<N$.
\end{oss}
We conclude this section by considering the localized versions of the spaces above. If $\Omega\subset\mathbb{R}^N$ is an open set, for $\mathbf{h}\in\mathbb{R}^N\setminus\{0\}$ we denote
\[
\Omega_{\mathbf h}=\{x\in \Omega\, :\, x+t\,\mathbf{h}\in\Omega \mbox{ for every } t\in[0,1]\}.
\]
For a function $\psi\in L^p(\Omega)$, we define
\begin{equation}
\label{Ohm}
[\psi]_{\mathfrak{n}^{t,p}_{\infty,i}(\Omega)}=\sup_{|h|>0} \left\|\frac{\delta_{h\mathbf{e}_i} \psi}{|h|^t}\right\|_{L^p(\Omega_{h\mathbf{e}_i})},\qquad i=1,\dots,N,
\end{equation}
and
\begin{equation}
\label{Ohmm}
[\psi]_{\mathfrak{b}^{t,p}_{\infty,i}(\Omega)}=\sup_{|h|>0} \left\|\frac{\delta^2_{h\mathbf{e}_i} \psi}{|h|^t}\right\|_{L^p(\Omega_{2h\mathbf{e}_i})},\qquad i=1,\dots,N.
\end{equation}
Accordingly, we introduce the anisotropic Besov-Nikol'ski\u{\i} spaces on $\Omega$ as
\[
\mathcal{N}^{\mathbf{t},p}_{\infty}(\Omega):=\left\{\psi\in L^p(\Omega)\, :\, \sum_{i=1}^N [\psi]_{\mathfrak{n}^{t_i,p}_{\infty,i}(\Omega)}<+\infty\right\},
\]
and
\[
\mathcal{B}^{\mathbf{t},p}_{\infty}(\Omega):=\left\{\psi\in L^p(\Omega)\, :\, \sum_{i=1}^N [\psi]_{\mathfrak{b}^{t_i,p}_{\infty,i}(\Omega)}<+\infty\right\}.
\]
Finally, we define
\[
\mathcal{N}^{\mathbf{t},p}_{\infty,{\rm loc}}(\Omega):=\Big\{\psi\in L^p_{{\rm loc}}(\Omega)\, :\, \psi \in\mathcal{N}^{\mathbf{t},p}_{\infty}(E) \mbox{ for every } E\Subset \Omega \Big\},
\]
and
\[
\mathcal{B}^{\mathbf{t},p}_{\infty,{\rm loc}}(\Omega):=\Big\{\psi\in L^p_{{\rm loc}}(\Omega)\, :\, \psi \in\mathcal{B}^{\mathbf{t},p}_{\infty}(E) \mbox{ for every } E\Subset \Omega \Big\}.
\]
\begin{oss}
\label{oss:ristretto}
As for the case of $\mathbb{R}^N$, the definitions of $\mathcal{N}^{\mathbf{t},p}_{\infty}(\Omega)$ and $\mathcal{B}^{\mathbf{t},p}_{\infty}(\Omega)$ do not change if we perform the supremum in \eqref{Ohm} and \eqref{Ohmm} over $0<|h|<h_0$ for some $h_0>0$.
\end{oss}
\begin{coro}
\label{coro:nikolskicoro}
Let $\Omega\subset \mathbb{R}^N$ be an open set. Under the assumptions of Theorem \ref{teo:nikolskiembedding} and with the same notations, we have 
\[
\mathcal{N}^{\mathbf{t},p}_{\infty,{\rm loc}}(\Omega)\subset\mathcal{B}^{\mathbf{t},p}_{\infty,{\rm loc}}(\Omega)\subset L^{p\,\chi}_{\rm loc}(\Omega),\qquad \mbox{ for every }\ 1\le \chi<\frac{\gamma}{\gamma-p}. 
\]
\end{coro}
\begin{proof}
Let $\psi\in \mathcal{N}^{\mathbf{t},p}_{\infty,{\rm }loc}(\Omega)$ and let $E\Subset \Omega$, we prove first that $\psi\in \mathcal{B}^{\mathbf{t},p}_{\infty}(E)$. By triangle inequality
\[
\begin{split}
\left\|\frac{\delta^2_{h\mathbf{e}_i} \psi}{|h|^t}\right\|_{L^p(E_{2h\mathbf{e}_i})}&\le \left\|\frac{\delta_{h\mathbf{e}_i} \psi_{h\mathbf{e}_i}}{|h|^t}\right\|_{L^p(E_{2h\mathbf{e}_i})}+\left\|\frac{\delta_{h\mathbf{e}_i} \psi}{|h|^t}\right\|_{L^p(E_{2h\mathbf{e}_i})}\\
&\le \left\|\frac{\delta_{h\mathbf{e}_i} \psi}{|h|^t}\right\|_{L^p(E_{h\mathbf{e}_i})}+\left\|\frac{\delta_{h\mathbf{e}_i} \psi}{|h|^t}\right\|_{L^p(E_{h\mathbf{e}_i})},
\end{split}
\]
where we used a simple change a variable and the inclusion $E_{2h\mathbf{e}_i}\subset E_{h\mathbf{e}_i}$. By taking the supremum over $h$, we get the first conclusion.
\vskip.2cm\noindent
Let $\psi\in \mathcal{B}^{\mathbf{t},p}_{\infty,{\rm loc}}(\Omega)$ and let $E\Subset \Omega$, we prove that $\psi\in L^{p\,\chi}(E)$. We set $d=\mathrm{dist}(E,\partial \Omega)>0$, then there exist $x_1,\dots,x_k\in E$ such that
\[
E\subset \bigcup_{j=1}^k B_\frac{d}{8}(x_j).
\]
It is sufficient to prove that $\psi\in L^{p\,\chi}(B_{d/8}(x_j))$ for every $j=1,\dots,k$. We fix one of these balls and omit to indicate the center $x_j$ for simplicity. We then take a standard cut-off function $\eta\in C^\infty_0(B_{d/4})\subset C^\infty_0(\Omega)$ such that $\eta\equiv 1$ on $B_{d/8}$. Then we observe that $\psi\,\eta\in \mathcal{B}^{\mathbf{t},p}_{\infty}(\mathbb{R}^N)$: indeed, by triangle inequality and \eqref{leibniz} for every $h\not =0$ such that $|h|<d/8$ we have
\[
\begin{split}
\left\|\frac{\delta^2_{h\mathbf{e}_i} (\psi\,\eta)}{|h|^{t_i}}\right\|_{L^p(\mathbb{R}^N)}&\le \left\|\frac{\delta^2_{h\mathbf{e}_i} \eta}{|h|^{t_i}}\,\psi\right\|_{L^p(\mathbb{R}^N)}+2\,\left\|\frac{\delta_{h\mathbf{e}_i}\eta_{h\mathbf{e}_i}}{|h|^{t_i}}\,\delta_{h\mathbf{e}_i}\psi\right\|_{L^p(\mathbb{R}^N)}+\left\|\eta_{2h\mathbf{e}_i}\frac{\delta^2_{h\mathbf{e}_i} \psi}{|h|^{t_i}}\right\|_{L^p(\mathbb{R}^N)}\\
&\le 4\,\left(\frac{d}{8}\right)^{1-t_i}\,\|\nabla \eta\|_{L^\infty}\,\|\psi\|_{L^p(B_{\frac{d}{2}})}+\left\|\frac{\delta^2_{h\mathbf{e}_i} \psi}{|h|^{t_i}}\right\|_{L^p(B_{\frac{d}{2}})},\qquad i=1,\dots,N,
\end{split}
\]
and the supremum of the latter over $0<|h|<d/8$ is finite, since $B_{d/2}\Subset \Omega$ by construction. By appealing to Lemma \ref{lm:facile}, we thus get $\psi\,\eta\in \mathcal{B}^{\mathbf{t},p}_\infty(\mathbb{R}^N)$. We can use Theorem \ref{teo:nikolskiembedding} and get $\psi\,\eta\in L^{p\,\chi}(\mathbb{R}^N)$. Since $\eta\equiv 1$ on $B_{d/8}$, this gives the desired result.
\end{proof}

\section{A general scheme for improving differentiability}
\label{sec:3}
In this section we consider a slightly more general framework, with respect to that of Theorem \ref{teo:sobolevpq}. Namely, we consider
a set of $C^2$ convex functions $g_i:\mathbb{R}\to\mathbb{R}^+$ such that
\begin{equation}
\label{growth}
\frac{1}{\mathcal{C}}\, (|s|-\delta)^{p_i-2}_+\le g_i''(s)\le \mathcal{C}\,(|s|^{p_i-2}+1),\qquad i=1,\dots,N,
\end{equation}
for some $\mathcal{C}\ge 1$, $\delta\ge 0$ and $2\le p_1\le \dots\le p_{N-1}\le p_N$.
\begin{oss}
Let us point out the following simple inequality that will be used in what follows: for every $a\le s\le b$, we have
\begin{equation}
\label{monotonaNO}
g_i''(s)\le \widetilde{\mathcal{C}}_i\,\Big(g_i''(a)+g_i''(b)+1\Big),
\end{equation}
for some $\widetilde{\mathcal{C}}_i=\widetilde{\mathcal{C}}_i(p_i,\delta)\ge 1$. This follows with elementary manipulations, by exploiting \eqref{growth}. We leave the details to the reader.
\end{oss}
We then consider $u\in W^{1,{\bf p}}_{\rm loc}(\Omega)$ a local minimizer of 
\[
\mathfrak{F}(u;\Omega')=\sum_{i=1}^N \int_{\Omega'} g_i(u_{x_i})\,dx+\int_{\Omega'} f\,u\,dx.
\] 
In particular, $u$ solves
\begin{equation}
\label{pharmonic}
\sum_{i=1}^N \int g'_i(u_{x_i})\,\varphi_{x_i}\,dx+\int f\,\varphi\,dx=0,
\end{equation}
for every $\varphi\in W^{1,{\mathbf{p}}}_0(\Omega')$ and every $\Omega'\Subset\Omega$. 
For every $i=1,\dots,N$, we define
\[
\mathcal{V}_i=V_i(u_{x_i}),\qquad \mbox{ where }\quad V_i(t)=\int_0^t \sqrt{g''_i(\tau)}\,d\tau.
\]
Our aim is to prove that every $\mathcal{V}_i$ enjoys some weak differentiability properties.
We start with the following result. 
\begin{prop}[Initial gain]
\label{lm:step1}
Let $2\le p_1\le \dots\le p_{N-1}\le p_N$ 
and let $f\in W^{1,\mathbf{p}'}_{\rm loc}(\Omega)$. We suppose that
\[
u\in L^\infty_{\rm loc}(\Omega).
\]
Then for every $i=1,\dots,N$ we have 
\[
\mathcal{V}_i\in \mathcal{N}^{\mathbf{t},2}_{\infty,{\rm loc}}(\Omega),\qquad \mbox{ where }\mathbf{t}=\left(\frac{p_1}{p_N},\dots,\frac{p_{N-1}}{p_N},1\right).
\]
\end{prop}
\begin{proof}
We take $B_{r_0}\Subset B_{R_0}\Subset\Omega$ a pair of concentric balls centered at $x_0$ and set 
\[
h_0=(R_0-r_0)/4\qquad \mbox{ and }\qquad R=\frac{R_0+r_0}{2}.
\] 
Then we pick $\varphi\in W^{1,{\mathbf{p}}}_0(B_R)$ that we extend it to zero on $\mathbb{R}^N\setminus B_R$. For every $0<|h|<h_0$ we can insert the test function $\varphi_{-h\mathbf{e}_j}(x)$ in \eqref{pharmonic}. With a simple change of variables we get
\begin{equation}
\label{decalée}
\sum_{i=1}^N \int_\Omega g_i'\left((u_{x_i})_{h\mathbf{e}_j}\right)\,\varphi_{x_i}\,dx=\int_\Omega f_{h\mathbf{e}_j}\,\varphi\,dx.
\end{equation}
By subtracting \eqref{pharmonic} and \eqref{decalée} and dividing by $|h|$, we thus get
\[
\sum_{i=1}^N \int_\Omega \left[\frac{g_i'\left((u_{x_i})_{h\mathbf{e}_j}\right)-g_i'(u_{x_i})}{|h|}\right]\,\varphi_{x_i}\,dx=\int_\Omega \frac{\delta_{h\mathbf{e}_j}f}{|h|}\,\varphi\,dx.
\]
We now make the following particular choice
\[
\varphi=\zeta^2\, \frac{\delta_{h\mathbf{e}_j} u}{|h|^{s_j}},
\]
where $s_j\in(-1,1]$ will be chosen below and $\zeta$ is the standard cut-off function 
\[
\zeta(x)=\min\left\{1,\left(\frac{R-|x-x_0|}{R-r_0}\right)_+\right\}.
\]
We obtain
\[
\begin{split}
\sum_{i=1}^N &\int \left[\frac{\delta_{h\mathbf{e}_j}g_i'(u_{x_i})}{|h|}\right]\,\frac{\delta_{h\mathbf{e}_j}u_{x_i}}{|h|^{s_j}}\,\zeta^2\,dx\\
&\le2\, \sum_{i=1}^N\int \left|\frac{\delta_{h\mathbf{e}_j}g_i'(u_{x_i})}{h}\right|\,|\zeta_{x_i}|\,\zeta\,\left|\frac{\delta_{h\mathbf{e}_j} u}{|h|^{s_j}}\right|\, dx+\int \left|\frac{\delta_{h\mathbf{e}_j} f}{h}\right| \,\left|\frac{\delta_{h\mathbf{e}_j}u}{|h|^{s_j}}\right|\,\zeta^2\,dx.
\end{split}
\]
Recalling the definition of $\mathcal{V}_i$, using \eqref{monotone} in the left-hand side and \eqref{lipschitz} (in combination with \eqref{monotonaNO}) in the right-hand side, we obtain
\[
\begin{split}
\sum_{i=1}^N \int \left|\frac{\delta_{h\mathbf{e}_j} \mathcal{V}_i}{|h|^\frac{s_j+1}{2}}\right|^2\,\zeta^2\, dx&\le C \sum_{i=1}^N\int \left|\frac{\delta_{h\mathbf{e}_j}\mathcal{V}_i}{|h|^\frac{s_j+1}{2}}\right|\,\left[\sqrt{g_i''\left((u_{x_i})_{h\mathbf{e}_j}\right)}+\sqrt{g_i''\left(u_{x_i}\right)}+1\right]\,|\zeta_{x_i}|\,\zeta\,\left|\frac{\delta_{h\mathbf{e}_j} u}{|h|^\frac{s_j+1}{2}}\right|\, dx\\
&+\int \left|\frac{\delta_{h\mathbf{e}_j} f}{h}\right| \,\left|\frac{\delta_{h\mathbf{e}_j}u}{|h|^{s_j}}\right|\,\zeta^2\,dx.
\end{split}
\]
If we use H\"older and Young inequalities in the right-hand side, we can absorb the higher-order term. Namely, since we have
\[
\begin{split}
\sum_{i=1}^N\int \left|\frac{\delta_{h\mathbf{e}_j} \mathcal{V}_i}{|h|^\frac{s_j+1}{2}}\right|\,&\left[\sqrt{g_i''\left((u_{x_i})_{h\mathbf{e}_j}\right)}+\sqrt{g_i''\left(u_{x_i}\right)}+1\right]\,|\zeta_{x_i}|\,\zeta\,\left|\frac{\delta_{h\mathbf{e}_j} u}{|h|^\frac{s_j+1}{2}}\right|\, dx\\  
&\le C\,\tau\,\sum_{i=1}^N\, \int \left|\frac{\delta_{h\mathbf{e}_j}\mathcal{V}_i}{|h|^\frac{s_j+1}{2}}\right|^2\, \zeta^2\, dx\\
&+\frac{C}{\tau}\, \sum_{i=1}^N\, \int \left[g_i''\left((u_{x_i})_{h\mathbf{e}_j}\right)+g_i''\left(u_{x_i}\right)+1\right]\,|\zeta_{x_i}|^2\,\left|\frac{\delta_{h\mathbf{e}_j}u}{|h|^\frac{s_j+1}{2}}\right|^2\, dx, 
\end{split}
\]
where $0<\tau<1$, by choosing $\tau$ small enough, we thus get
\[
\begin{split}
\sum_{i=1}^N \int \left|\frac{\delta_{h\mathbf{e}_j}\mathcal{V}_i}{|h|^\frac{s_j+1}{2}}\right|^2\,\zeta^2\, dx&\le C\, \sum_{i=1}^N\, \int \left[g_i''\left((u_{x_i})_{h\mathbf{e}_j}\right)+g_i''\left(u_{x_i}\right)+1\right]\,|\zeta_{x_i}|^2\,\left|\frac{\delta_{h\mathbf{e}_j}u}{|h|^\frac{s_j+1}{2}}\right|^2\, dx\\
&+C\,\left(\int_{B_R} \left|\frac{\delta_{h\mathbf{e}_j}f}{|h|^\frac{s_j+1}{2}}\right|^{p_j'}\,dx\right)^\frac{1}{p'_j}\,\left(\int_{B_R} \left|\frac{\delta_{h\mathbf{e}_j}u}{|h|^\frac{s_j+1}{2}}\right|^{p_j}\,dx\right)^\frac{1}{p_j}.
\end{split}
\]
By basic properties of differential quotients, we get for $0<|h|<h_0$
\[
\begin{split}
\int_{B_R} \,\left|\frac{\delta_{h\mathbf{e}_j}u}{|h|^\frac{s_j+1}{2}}\right|^{p_j}\,dx&\le C\, h_0^{\frac{1-s_j}{2}\,p_j}\,\int_{B_{R_0}} |u_{x_j}|^{p_j}\,dx,
\end{split}
\]
and similarly
\[
\int_{B_R} \left|\frac{\delta_{h\mathbf{e}_j} f}{|h|^\frac{s_j+1}{2}}\right|^{p_j'}\,dx\le C\,h_0^{\frac{1-s_j}{2}\,p_j'}\,\int_{B_{R_0}} |f_{x_j}|^{p_j'}\, dx.
\]
This yields
\begin{equation}
\label{intermediobis}
\begin{split}
\sum_{i=1}^N \int \left|\frac{\delta_{h\mathbf{e}_j}\mathcal{V}_i}{|h|^\frac{s_j+1}{2}}\right|^2\,\zeta^2\, dx&\le \frac{C}{(R_0-r_0)^2}\, \sum_{i=1}^N\, \int \left[g_i''\left((u_{x_i})_{h\mathbf{e}_j}\right)+g_i''\left(u_{x_i}\right)+1\right]\,\left|\frac{\delta_{h\mathbf{e}_j}u}{|h|^\frac{s_j+1}{2}}\right|^2\, dx\\
&+C\,h_0^{1-s_j}\left(\int_{B_{R_0}} \left|f_{x_j}\right|^{p_j'}\,dx\right)^\frac{1}{p'_j}\,\left(\int_{B_{R_0}} \left|u_{x_j}\right|^{p_j}\,dx\right)^\frac{1}{p_j}.
\end{split}
\end{equation}
We use again H\"older inequality in the first term in the right-hand side, so that
\[
\begin{split}
\sum_{i=1}^N &\int \left[g_i''\left((u_{x_i})_{h\mathbf{e}_j}\right)+g_i''\left(u_{x_i}\right)+1\right]\,\left|\frac{\delta_{h\mathbf{e}_j}u}{|h|^\frac{s_j+1}{2}}\right|^2\, dx\\
&\le \sum_{i=1}^N\, \left(\int_{B_R} \left[g_i''\left((u_{x_i})_{h\mathbf{e}_j}\right)+g_i''\left(u_{x_i}\right)+1\right]^\frac{p_i}{p_i-2}\,dx\right)^\frac{p_i-2}{p_i}\,\left(\int_{B_R} \,\left|\frac{\delta_{h\mathbf{e}_j}u}{|h|^\frac{s_j+1}{2}}\right|^{p_i}\,dx\right)^\frac{2}{p_i}.
\end{split}
\]
We now observe that with simple manipulations we have
\[
\begin{split}
\left(\int_{B_R} \left[g_i''\left((u_{x_i})_{h\mathbf{e}_j}\right)+g_i''\left(u_{x_i}\right)+1\right]^\frac{p_i}{p_i-2}\,dx\right)^\frac{p_i-2}{p_i}
&\le C\,\left(\int_{B_{R_0}} |g_i''\left(u_{x_i}\right)+1|^\frac{p_i}{p_i-2}\,dx\right)^\frac{p_i-2}{p_i},
\end{split}
\]
since for every $0<|h|<h_0$ we have $B_{R}+h\mathbf{e}_j\subset B_{R_0}$, by construction.
Thus from \eqref{intermediobis} we obtain
\begin{equation}
\label{attrezzobis}
\begin{split}
\sum_{i=1}^N \int \left|\frac{\delta_{h\mathbf{e}_j}\mathcal{V}_i}{|h|^\frac{s_j+1}{2}}\right|^2\,\zeta^2\, dx&\le \frac{C}{(R_0-r_0)^2}\, \sum_{i=1}^N\, \left\|g_i''\left(u_{x_i}\right)+1\right\|_{L^\frac{p_i}{p_i-2}(B_{R_0})}\left(\int_{B_{R}}\left|\frac{\delta_{h\mathbf{e}_j}u}{|h|^\frac{s_j+1}{2}}\right|^{p_i}\, dx\right)^\frac{2}{p_i}\\
&+C\,h_0^{1-s_j}\left(\int_{B_{R_0}} \left|f_{x_j}\right|^{p_j'}\,dx\right)^\frac{1}{p'_j}\,\left(\int_{B_{R_0}} \left|u_{x_j}\right|^{p_j}\,dx\right)^\frac{1}{p_j}.
\end{split}
\end{equation}
The first term in the right-hand side is more delicate and we have to distinguish between two cases.
\vskip.2cm\noindent
{\bf Case A: $j=N$}. By hypothesis we have $p_i\le p_N$ for every $1\le i\le N$. Thus we get
\[
\int_{B_R} \,\left|\frac{\delta_{h\mathbf{e}_N}u}{|h|^\frac{s_N+1}{2}}\right|^{p_i}\,dx\le C\, R^{N\,\frac{p_N-p_i}{p_N}}\,\left(\int_{B_R} \,\left|\frac{\delta_{h\mathbf{e}_N}u}{|h|^\frac{s_N+1}{2}}\right|^{p_N}\,dx\right)^\frac{p_i}{p_N},\qquad i=1,\dots,N.
\]
We can then choose $s_N=1$ so that $(s_N+1)/2=1$ as well. 
Then from \eqref{attrezzobis} we get
\[
\begin{split}
\sum_{i=1}^N \int \left|\frac{\delta_{h\mathbf{e}_N}(\mathcal{V}_i)}{h}\right|^2 \zeta^2 \,dx & \le \frac{C}{(R_0-r_0)^2}\, \left[\sum_{i=1}^N\,R_0^{2\,N\,\frac{p_N-p_i}{p_N\,p_i}}\, \left\|g_i''\left(u_{x_i}\right)+1\right\|_{L^\frac{p_i}{p_i-2}(B_{R_0})}\right] \\ 
 & \times \left(\int_{B_{R_0}} \left|u_{x_N}\right|^{p_N}\,dx\right)^\frac{2}{p_N}+C\,\left(\int_{B_{R_0}} |f_{x_N}|^{p_N'}\, dx\right)^\frac{1}{p'_N}\,\left(\int_{B_{R_0}} |u_{x_N}|^{p_N}\right)^\frac{1}{p_N}.
\end{split}
\]
\vskip.2cm\noindent
{\bf Case B: $1\le j\le N-1$}. This in turn has to be divided in two sub-cases.
\vskip.2cm\noindent
\underline{\it Case B.1: $1\le i\le j$.} This is similar to Case A, since by hypothesis we have $p_i\le p_j$. Then for $0<|h|<h_0$ we simply have
\[
\begin{split}
\int_{B_R} \,\left|\frac{\delta_{h\mathbf{e}_j} u}{|h|^\frac{s_j+1}{2}}\right|^{p_i}\,dx&\le C\, R^{N\,\frac{p_j-p_i}{p_j}}\,\left(\int_{B_R} \,\left|\frac{\delta_{h\mathbf{e}_j}u}{|h|^\frac{s_j+1}{2}}\right|^{p_j}\,dx\right)^\frac{p_i}{p_j}\\
&\le C\, h_0^{\frac{1-s_j}{2}\,p_i}R^{N\,\frac{p_j-p_i}{p_j}}\left(\int_{B_{R_0}} \,|u_{x_j}|^{p_j}\,dx\right)^\frac{p_i}{p_j}.
\end{split}
\]
\underline{\it Case B.2: $j+1\le i\le N$.} Here we should be more careful. The order of maximal differentiability $t_j=(s_j+1)/2$ is determined here. We set $t_j=p_j/p_N$ as in the statement, we thus get
\[
\begin{split}
\int_{B_R} \,\left|\frac{\delta_{h\mathbf{e}_j}u}{|h|^{t_j}}\right|^{p_i}\,dx
&\le \int_{B_R} \frac{\left|\delta_{h\mathbf{e}_j}u\right|^{p_j}}{|h|^{{t_j}\,p_i}}\,dx\,\left\|\delta_{h\mathbf{e}_j}u\right\|_{L^\infty(B_R)}^{p_i-p_j}.
\end{split}
\]
Since $p_j-t_j\,p_i\ge 0$, we further observe that for every $0<|h|<h_0$ we have
\[
\begin{split}
\int_{B_R} \frac{\left|\delta_{h\mathbf{e}_j}u\right|^{p_j}}{|h|^{{t_j}\,p_i}}\,dx\le h_0^{p_j-t_j\,p_i}\,\int_{B_R} \left|\frac{\delta_{h\mathbf{e}_j}u}{h}\right|^{p_j}\,dx&\le C\,h_0^{p_j-t_j\,p_i}\, \int_{B_{R_0}} |u_{x_j}|^{p_j}\,dx.
\end{split}
\]
Moreover
\[
\|\delta_{h\mathbf{e}_j}u\|_{L^\infty(B_R)}\le 2\, \|u\|_{L^\infty(B_{R_0})}.
\]
By using the previous estimates in \eqref{attrezzobis} 
we thus obtain\footnote{It is intended that the second term in the right-hand side is $0$ for $j=N$.}
\[
\begin{split}
\sum_{i=1}^N &\int \left|\frac{\delta_{h\mathbf{e}_j}\mathcal{V}_i}{|h|^{t_j}}\right|^2\,\zeta^2\, dx\\
&\le \frac{C\,h_0^{2\,(1-t_j)}}{(R_0-r_0)^2}\,\left[\sum_{i=1}^j R_0^{2\,N\,\frac{p_j-p_i}{p_j\,p_i}}\,\left\|g_i''\left(u_{x_i}\right)+1\right\|_{L^\frac{p_i}{p_i-2}(B_{R_0})}\right]\|u_{x_j}\|_{L^{p_j}(B_{R_0})}^2\\
&+\frac{C}{(R_0-r_0)^2}\,\left[\sum_{i=j+1}^N h_0^{2\,\left(\frac{p_j}{p_i}-\frac{p_j}{p_N}\right)}\,\left\|g_i''\left(u_{x_i}\right)+1\right\|_{L^\frac{p_i}{p_i-2}(B_{R_0})}\,\|u\|_{L^\infty(B_{R_0})}^{2\,\left(1-\frac{p_j}{p_i}\right)}\right]\,\left\|u_{x_j}\right\|_{L^{p_j}(B_{R_0})}^{2\,\frac{p_j}{p_i}}\\
&+C\,h_0^{2\,(1-t_j)}\left\|f_{x_j}\right\|_{L^{p_j'}(B_{R_0})}\,\left\|u_{x_j}\right\|_{L^{p_j}(B_{R_0})},
\end{split}
\]
for a constant $C=C(N,p_1,\dots,p_N)>0$.  By taking the supremum over $0<|h|<h_0$,  summing over $j=1,\dots,N$ and recalling that $\zeta\equiv 1$ on $B_r$, we finally conclude that 
\[
\sum_{j=1}^N \sup_{0<|h|<h_0} \left\|\frac{\delta_{h\mathbf{e}_j}\mathcal{V}_i}{|h|^{t_j}}\right\|_{L^2(B_r)}<+\infty,\qquad i=1,\dots,N.
\]
We now take $E\Subset\Omega$ such that $d=\mathrm{dist}(E,\partial \Omega)>0$. There exist $J\in\mathbb{N}$ and $x_1,\dots,x_J\in E$ such that
\[
E\subset \bigcup_{k=1}^J B_\frac{d}{4}(x_k).
\]
By observing that each set $E_{h\mathbf{e}_j}$ is still covered by this family of balls, we thus obtain
\[
\sum_{j=1}^N\sup_{0<|h|<\frac{d}{4}} \left\|\frac{\delta_{h\mathbf{e}_j}\mathcal{V}_i}{|h|^{t_j}}\right\|_{L^2(E_{h\mathbf{e}_j})}\le\sum_{j=1}^N \sum_{k=1}^J \sup_{0<|h|<\frac{d}{4}} \left\|\frac{\delta_{h\mathbf{e}_j}\mathcal{V}_i}{|h|^{t_j}}\right\|_{L^2(B_\frac{d}{4}(x_k))}<+\infty,\qquad i=1,\dots,N.
\]
By taking into account Remark \ref{oss:ristretto}, this gives $\mathcal{V}_i\in \mathcal{N}^{\mathbf{t},2}_{\infty,{\rm loc}}(\Omega)$, as desired.
\end{proof}
By using Corollary \ref{coro:nikolskicoro}, we also get the following higher integrability result.
\begin{coro}
\label{coro:colpointegrabile}
Under the previous assumptions, for every $i=1,\dots,N$ we have 
\[
\mathcal{V}_i\in L^{2\,\chi}_{\rm loc}(\Omega),\qquad \mbox{ for every } 1\le\chi<\frac{\gamma}{\gamma-2},\ \mbox{ where } \gamma=\sum_{j=1}^N\frac{1}{t_j}=p_N\,\frac{N}{\overline p}.
\]
\end{coro}
The next result shows that each time $\mathcal{V}_1,\dots,\mathcal{V}_N$ gain integrability, then we can improve their differentiability as well.
\begin{prop}[Improvement of differentiability]
\label{lm:stepk}
Let us suppose $u\in L^\infty_{\rm loc}(\Omega)$ and $\mathcal{V}_1,\dots,\mathcal{V}_N\in L^{2\,\chi}_{\rm loc}(\Omega)$, for some $\chi>1$.
Then we have 
\[
\mathcal{V}_i\in\mathcal{N}^{\mathbf{r},2}_{\infty,{\rm loc}}(\Omega),\qquad i=1,\dots,N,
\]
where the vector $\mathbf{r}=(r_1,\dots,r_N)$ is given by
\begin{equation}
\label{rj}
r_j=\min\left\{\frac{p_j}{p_N}+\frac{p_j}{2}\,(\chi-1),\, 1\right\},\qquad j=1,\dots,N.
\end{equation}
\end{prop} 
\begin{proof}
We first observe that the hypothesis on $\mathcal{V}_i$ implies that $u_{x_i}\in L^{p_i\,\chi}_{\rm loc}(\Omega)$, thanks to \eqref{growth}. Moreover, for $j=N$ by Proposition \ref{lm:step1} we already know that we have maximal differentiability, i.e. $r_N=1$.
\par
Let us fix $1\le j\le N-1$, we go back to \eqref{intermediobis} and we use H\"older inequality in the right-hand side for the terms $i\ge j+1$, with exponents 
\[
\frac{p_i\,\chi}{p_i-2}\qquad \mbox{ and }\qquad \frac{p_i\,\chi}{p_i\,(\chi-1)+2}.
\] 
This gives
\begin{equation}
\label{45}
\begin{split}
\sum_{i=1}^N \int \left|\frac{\delta_{h\mathbf{e}_j}\mathcal{V}_i}{|h|^\frac{s_j+1}{2}}\right|^2\, \zeta^2\,dx&\le \frac{C}{(R_0-r_0)^2}\,\sum_{i=1}^j\, \left(\int_{B_R} \left[g_i''\left((u_{x_i})_{h\mathbf{e}_j}\right)+g_i''\left(u_{x_i}\right)+1\right]^\frac{p_i}{p_i-2}\,dx\right)^\frac{p_i-2}{p_i}\\
&\times\left(\int_{B_R} \,\left|\frac{\delta_{h\mathbf{e}_j}u}{|h|^\frac{s_j+1}{2}}\right|^{p_i}\,dx\right)^\frac{2}{p_i}\\
&+ \frac{C}{(R_0-r_0)^2}\,\sum_{i=j+1}^N\, \left(\int_{B_R} \left[g_i''((u_{h\mathbf{e}_j})_{x_i})+g_i''(u_{x_i})+1\right]^\frac{\chi\,p_i}{p_i-2}\,dx\right)^\frac{p_i-2}{\chi\, p_i}\\
&\times\left(\int_{B_R} \,\left|\frac{\delta_{h\mathbf{e}_j}u}{|h|^\frac{s_j+1}{2}}\right|^{\frac{2\,\chi\, p_i}{p_i\,(\chi-1)+2}}\,dx\right)^{\frac{p_i\,(\chi-1)+2}{\chi\, p_i}}\\
&+C\,h_0^{1-s_j}\,\left\|f_{x_j}\right\|_{L^{p_j'}(B_{R_0})}\,\left\|u_{x_j}\right\|_{L^{p_j}(B_{R_0})}.
\end{split}
\end{equation}
The first sum on the right-hand side is estimated as in Proposition \ref{lm:step1}. For the second one, we have to make two separate discussione, depending on whether
\begin{itemize}
\item $\chi$ is such that
\begin{equation}
\label{favorevole}
\chi\ge1+2\,\left(\frac{1}{p_j}-\frac{1}{p_N}\right);
\end{equation}
\vskip.2cm
\item or $\chi$ is such that
\begin{equation}
\label{sfavorevole}
\chi <1+2\,\left(\frac{1}{p_j}-\frac{1}{p_N}\right).
\end{equation}
\end{itemize}

\vskip.2cm\noindent
If we assume that \eqref{favorevole} is satisfied, then  we have as well
\[
\chi\ge1+2\,\left(\frac{1}{p_j}-\frac{1}{p_i}\right),\qquad \mbox{ for every }i=1,\dots,N.
\]
that is
\[
\frac{2\,p_i}{p_i\,(\chi-1)+2}\le p_j,\qquad \mbox{ for every } i=1,\dots,N.
\]
Back to \eqref{45}, we can choose $s_j=1$ and we simply have
\[
\begin{split}
\int_{B_R} \,\left|\frac{\delta_{h\mathbf{e}_j}u}{h}\right|^{\frac{2\,\chi\, p_i}{p_i\,(\chi-1)+2}}\,dx
&\le C\, R^{N\,\left(1-\frac{2\,p_i}{p_j\,[p_i\,(\chi-1)+2]}\right)}\,\left(\int_{B_{R_0}} \,|u_{x_j}|^{\chi\,p_j}\,dx\right)^\frac{2\,p_i}{p_j[p_i(\chi-1)+2]},
\end{split}
\]
thus with the usual manipulations we obtain
\[
\begin{split}
\sum_{i=1}^N& \int \left|\frac{\delta_{h\mathbf{e}_j}\mathcal{V}_i}{h}\right|^2\,\zeta^2\, dx\\
&\le \frac{C}{(R_0-r_0)^2}\,\left[\sum_{i=1}^j R_0^{2\,N\,\frac{p_j-p_i}{p_j\,p_i}}\,\left\|g_i''\left(u_{x_i}\right)+1\right\|_{L^\frac{p_i}{p_i-2}(B_{R_0})}\right]\|u_{x_j}\|_{L^{p_j}(B_{R_0})}^2\\
&+ \frac{C}{(R_0-r_0)^2}\,\left[\sum_{i=j+1}^N R_0^{N\,\left(\frac{\chi-1}{\chi}+\frac{2\,(p_j-p_i)}{\chi\,p_jp_i}\right)}\,\|g_i''(u_{x_i})+1\|_{L^\frac{\chi\,p_i}{p_i-2}(B_{R_0})}\right] \|u_{x_j}\|_{L^{\chi\,p_j}(B_{R_0})}^2\\
&+C\,\left\|f_{x_j}\right\|_{L^{p_j'}(B_{R_0})}\,\left\|u_{x_j}\right\|_{L^{p_j}(B_{R_0})}.
\end{split}
\]
Let us now consider the case where \eqref{sfavorevole} is verified. 
In this case, by using that $u\in L^\infty_{\rm loc}$, if we set 
$$
r_j=\frac{1+s_j}{2}=\frac{p_j}{p_N}+\frac{p_j}{2}\,(\chi-1)<1,
$$
we obtain
\[
\begin{split}
\int_{B_R} \,\left|\frac{\delta_{h\mathbf{e}_j}u}{|h|^{r_j}}\right|^{\frac{2\,\chi\, p_i}{p_i\,(\chi-1)+2}}\,dx&\le\|\delta_{h\mathbf{e}_j}u\|^{\frac{2\,\chi\, p_i}{p_i\,(\chi-1)+2}-\chi\, p_j}_{L^\infty(B_R)}\, \int_{B_R} \frac{|\delta_{h\mathbf{e}_j}u|^{p_j\,\chi}}{|h|^{r_j\frac{2\,\chi\, p_i}{p_i\,(\chi-1)+2}}}\,dx.
\end{split}
\]
We observe that by construction
\[
r_j\,\frac{2\,\chi\,p_i}{p_i\,(\chi-1)+2}\le \chi\,p_j.
\]
Then as before, we obtain for $0<|h|<h_0$,
\[
\begin{split}
\sum_{i=1}^N &\int \left|\frac{\delta_{h\mathbf{e}_j}\mathcal{V}_i}{|h|^{r_j}}\right|^2\, \zeta^2\,dx\\
&\le \frac{C\,h_0^{2\,(1-r_j)}}{(R_0-r_0)^2}\,\left[\sum_{i=1}^j R_0^{2\,N\,\frac{p_j-p_i}{p_j\,p_i}}\,\left\|g_i''\left(u_{x_i}\right)+1\right\|_{L^\frac{p_i}{p_i-2}(B_{R_0})}\right]\|u_{x_j}\|_{L^{p_j}(B_{R_0})}^2\\
&+\frac{C}{(R_0-r_0)^2}\,\left[\sum_{i=1}^N\left\|g_i''(u_{x_i})+1\right\|_{L^\frac{\chi\,p_i}{p_i-2}(B_{R_0})}\,h_0^{2\,\left(\frac{p_j}{p_i}-\frac{p_j}{p_N}\right)}\,\|u\|^{2\,\left(1-\frac{p_j}{p_i}\right)-p_j\,(\chi-1)}_{L^\infty(B_{R_0})}\right]\|u_{x_j}\|_{L^{\chi\,p_j}(B_{R_0})}^{2\,\frac{p_j}{p_i}+p_j\,(\chi-1)}\\
&+C\,h_0^{2\,(1-r_j)}\left\|f_{x_j}\right\|_{L^{p_j'}(B_{R_0})}\,\left\|u_{x_j}\right\|_{L^{p_j}(B_{R_0})}.
\end{split}
\]
Thus, from the previous estimate, we get $\mathcal{V}_i\in\mathcal{N}^{\mathbf{r},2}_{\infty,{\rm loc}}(\Omega)$ by proceeding as in the final part of Proposition \ref{lm:step1}.
\end{proof}

Again by Corollary \ref{coro:nikolskicoro}, we also get the following.
\begin{coro}
\label{coro:stepk}
Under the previous assumptions, for every $i=1,\dots,N$ we have 
\[
\mathcal{V}_N\in L^{2\,\chi}_{\rm loc}(\Omega)\ \Longrightarrow\ \mathcal{V}_i\in L^{2\,\vartheta}_{\rm loc}(\Omega),\qquad \mbox{ for every } 1\le\vartheta<\frac{\gamma}{\gamma-2},\ \mbox{ where } \gamma=\sum_{j=1}^N\frac{1}{r_j},
\]
and $r_j$ is defined in \eqref{rj}.
\end{coro}

\section{Local Sobolev estimate in a particular case}
\label{sec:4}

We now specialize the discussion to the situation where we just have two growth exponents $2\le p<q$. Namely,
let $\ell\in\{1,2,\cdots, N-1\}$ and consider 
\[
p_1=\cdots=p_\ell=p<p_{\ell+1}=\cdots =p_N=q,\qquad \mbox{ with } p\ge 2,
\]
as in the statement of Theorem \ref{teo:sobolevpq}.

\begin{proof}[Proof of Theorem \ref{teo:sobolevpq}]
Let us set
\begin{equation}
\label{tau0}
\tau_0:=1-\frac{1}{N-1}\,\frac{p}{q},
\end{equation}
and observe that $0<\tau_0<1$. We take $\{\alpha_k\}$ an increasing sequence of positive numbers with 
\[
1>\alpha_k>\tau_0,\ \mbox{ for every }k\in\mathbb{N},\qquad\lim_{k\to \infty} \alpha_k=1.
\]
Let $i=1,\cdots, N$, by Proposition \ref{lm:step1} we have $\mathcal{V}_i\in \mathcal{N}^{\mathbf{t}_0,2}_{\infty,{\rm loc}}(\Omega)$, where
\[
\mathbf{t}_0=(\underbrace{t_0,\cdots,t_0}_\ell,1,\cdots,1)=\left(\underbrace{p/q,\dots,p/q}_\ell,1,\cdots, 1\right).
\]
Moreover, if we set 
\[
\gamma_0=\frac{q}{p}\,\ell+N-\ell\qquad \mbox{ and }\qquad  \chi_0 =1+\alpha_0\,\frac{2}{\gamma_0-2},
\]
we have $\mathcal{V}_i\in L^{2\,\chi_0}_{\rm loc} (\Omega)$ by Corollary \ref{coro:colpointegrabile}. We now repeatedly apply Proposition \ref{lm:stepk} and Corollary \ref{coro:stepk}: after $k+1$ steps, we get  $\mathcal{V}_i\in \mathcal{N}^{\mathbf{t}_{k},2}_{\infty,{\rm loc}}(\Omega)$
where 
\[
\mathbf{t}_{k}=(\underbrace{t_{k},\dots,t_{k}}_\ell,1,\cdots,1)\qquad \mbox{ with } 
t_{k}=\min\left\{\frac{p}{q}+\frac{p}{2}\,(\chi_{k-1}-1),1\right\},
\]
and
\begin{equation}
\label{magheooo}
\chi_{k-1}=1+\alpha_{k-1}\,\frac{2}{\gamma_{k-1}-2},\qquad \gamma_{k-1}=\frac{\ell}{t_{k-1}}+N-\ell. 
\end{equation}
We want to prove that under the standing assumptions \eqref{condtotN1} or \eqref{condtot}, there exists $k_0\in\mathbb{N}$ such that
\[
\frac{p}{q}+\frac{p}{2}\,(\chi_{k_0-1}-1)\ge 1.
\]
By using the relations \eqref{magheooo}, this is the same as
\begin{equation}
\label{desire}
\frac{p}{q}+\alpha_{k_0-1}\,\frac{p}{\dfrac{\ell}{t_{k_0-1}}+N-\ell-2}\ge 1.
\end{equation}
Until this does not occur, we thus have that $\{t_{k}\}_{k\in\mathbb{N}}$ coincides with the recursive sequence defined by
\begin{equation}
\label{tk}
\left\{\begin{array}{lcl}
t_0&=&\dfrac{p}{q}\\
&&\\
t_{k+1}&=&\dfrac{p}{q}+\alpha_k\,b(t_k),
\end{array}
\right.
\end{equation}
where the function $t\mapsto b(t)$ is defined by
\[
b(t)=\frac{p}{\dfrac{\ell}{t}+N-2-\ell},\qquad \mbox{ for } t>0\ \mbox { and }\ t\not=\frac{\ell}{\ell-(N-2)}.
\]
 We  observe that, for any $\ell\in\{1,2,\cdots, N-1\}$, $b(t)$ is an increasing function on its domain and it is positive for $t$ in the interval $(0,N-1)$\footnote{Indeed, for $\ell \le N-2$, $b(t)$ is positive increasing for $t>0$.}.
\par
In order to obtain \eqref{desire} and conclude the proof, we consider two possibilities for the sequence \eqref{tk}:
\begin{enumerate}
\item[]{\bf Alternative} I) either there exists $k_0$ such that $t_{k_0}\ge N-1$;
\vskip.2cm
\item[]{\bf Alternative} II) or $t_k<N-1$ for every $k\in\mathbb{N}$.
\end{enumerate}
\vskip.2cm
If {\bf Alternative} I) occurs the proof ends, since we automatically get \eqref{desire} and we can stop the process at $t_{k_0}$. 
\vskip.2cm\noindent
In {\bf Alternative} II), using the monotone behaviour of $b$ and $\{\alpha_k\}_{k\in\mathbb{N}}$, we get that $\{t_k\}_{k\in\mathbb{N}}$ is an increasing sequence, thus it admits a limit $L$ with
\begin{equation}
\label{LcaseII}
\frac{p}{q}<L\le N-1.
\end{equation}
In order to obtain \eqref{desire} and conclude the proof, it would be sufficient to show that $L>1$.
By recalling that $\{\alpha_k\}_{k\in\mathbb{N}}$ converges to $1$ by construction, the possible limits $L$ of $\{t_k\}_{k\in\mathbb{N}}$ can be found among the solutions of the equation
\begin{equation}
\label{limitk}
L= \frac{p}{q}+\frac{p\,L}{\ell+L\,(N-\ell-2)}.
\end{equation}  
$\boxed{\mbox{Case } \ell=N-2}$ In this case (which can happen only for $N\ge 3$), the equation \eqref{limitk} is linear and we immediately get
\[
L=\frac{p}{q}\,\frac{N-2}{N-2-p}.
\]
This implies that if $N-2\le p$ we are indeed in {\bf Alternative} I), since we violate\footnote{In this case the sequence $\{t_k\}$ diverges to $+\infty$} \eqref{LcaseII}. If on the other hand $N-2>p$, then
$L>1$ thanks to hypothesis \eqref{condtot}.
\vskip.2cm\noindent
$\boxed{\mbox{Case } 1\le \ell\le N-3}$  Observe that 
this can happen only for $N\ge 4$. From \eqref{limitk} we get that the possible limits of $t_k$ are determined by the roots of the polynomial:
\begin{equation}
\label{polinomio}
P(t)=t^2(N-2-\ell)-t\,\left[(N-2-\ell)\,\frac{p}{q}+p-\ell\right]-\frac{p}{q}\,\ell.
\end{equation}
By a simple computation, we see that $P$ has real roots $L_1\le L_2$ if and only if 
\begin{equation}
\label{deltapos}
(N-2-\ell)\,\frac{p}{q}+(\sqrt{\ell}-\sqrt{p})^2\ge 0.
\end{equation}
Since $\ell\le N-3$ the previous condition is always satisfied (with strict inequality sign, indeed). We have 
\[
P(t)<0 \qquad \Longleftrightarrow \qquad L_1<t<L_2.
\] 
If we observe that $P(p/q)=-p^2/q<0$, we thus get
\[
 L_1<\frac{p}{q}<L_2.
\] 
Since $\{t_k\}_{k\in\mathbb{N}}$ is increasing and $t_0=p/q$, this implies
\[
\lim_{k\to\infty} t_k=L_2.
\]
We now observe that we have 
\[
L_2>1\quad \Longleftrightarrow\quad P(1)<0 \quad \Longleftrightarrow\quad \mbox{ hypothesis }\eqref{condtot}.
\]
and we are done.
\vskip.2cm\noindent
$\boxed{\mbox{Case } \ell=N-1}$. This case is subtler. Let us start by looking at the subcase $p\ge N-1$.
\vskip.2cm\noindent
{\bf Case $p\ge N-1$.}
We first recall that
\[
t_{k+1}-t_k=\frac{p}{q}+\Big(\alpha_k\,b(t_k)-t_k\Big).
\]
Then observe that the function (recall the definition \eqref{tau0} of $\tau_0$)
\[
\varphi(t)=\tau_0^2\,b(t)-t,\qquad t\in\left[\frac{p}{q},N-1\right),
\]
is such that
\[
\varphi'(t)=\frac{\tau_0^2\,p\,(N-1)}{(N-1-t)^2}-1\ge 0\quad \Longleftrightarrow \quad N-1>t\ge N-1-\tau_0\,\sqrt{(N-1)\,p}=:\tilde {t}.
\]
Since we are supposing $p\ge N-1$, the choice of $\tau_0$ entails
\[
{\tilde t}\le N-1-\tau_0\,(N-1)= \frac{p}{q}.
\]
This implies that if $p\ge N-1$, then $\varphi$ is strictly increasing on $[p/q,N-1)$. By recalling that $\alpha_k>\tau_0>\tau_0^2$ and $t_0=p/q$
we get
\[
t_{k+1}-t_k> \frac{p}{q}+\Big(\tau_0^2\,b(t_0)-t_0\Big)=\tau_0^2\,b\left(\frac{p}{q}\right)>0,
\]
thus the sequence can not converge to a finite value. This means that in this case we are indeed in {\bf Alternative} I) and thus we are done.
\par
Observe in particular that since by assumption $p\ge 2$, the previous discussion implies that for $N=2$ and $N=3$ we finished the proof.
\vskip.2cm\noindent
{\bf Case $2\le p< N-1$ and $N\ge 4$.}
Again, the possible limits of $\{t_k\}_{n\in\mathbb{N}}$ are given by the roots of the polynomial $P$ defined in \eqref{polinomio}. We first observe that condition \eqref{deltapos} now reads \begin{equation}
\label{delta}
\dfrac{p}{\left(\sqrt{N-1}-\sqrt{p}\right)^2}\le q.
\end{equation}
When this is fulfilled, $P$ admits real roots. 
\par
We can thus observe that if $p$ and $q$ satisfy the third block of assumptions in \eqref{condtotN1}, $P$ has not real roots which implies that in this case we are in the situation I) and the proof is over.
\par
We assume that \eqref{delta} is verified. In this case we have 
\[
P(t)<0 \qquad \Longleftrightarrow \qquad t<L_1\ \mbox{ or }\ t>L_2.
\]
We still have $P(p/q)<0$, so that
\begin{equation}
\label{pozzetto}
\frac{p}{q}<L_1<L_2\qquad \mbox{ or }\qquad L_1<L_2<\frac{p}{q}.
\end{equation}
Since $\{t_k\}_{k\in\mathbb{N}}$ is increasing and $t_0=p/q$, this implies
\[
\lim_{k\to\infty} t_k=L_1,
\]
and thus the second alternative in \eqref{pozzetto} is ruled out. 
We compute $L_1$, this is given by
\[
L_1=\frac{N-1-\dfrac{p}{q'}-\sqrt{\left(N-1-\dfrac{p}{q'}\right)^2-4\,(N-1)\,\dfrac{p}{q}}}{2}.
\]
Observe that
\[
\begin{split}
L_1>1 &\quad \Longleftrightarrow\quad N-3-\dfrac{p}{q'}>\sqrt{\left(N-1-\dfrac{p}{q'}\right)^2-4\,(N-1)\,\dfrac{p}{q}}.
\end{split}
\]
A necessary condition for this to happen is that
\[
N-3>\frac{p}{q'}\quad \Longleftrightarrow\quad p\le N-3\quad \cup\quad \left\{\begin{array}{rcl}
p&>&N-3,\\
&&\\
q&<&\dfrac{p}{p-(N-3)}.
\end{array}
\right.
\]
When these conditions are in force, then we obtain
\[
\left(N-1-\dfrac{p}{q'}\right)^2+4-4\,\left(N-1-\dfrac{p}{q'}\right)>\left(N-1-\dfrac{p}{q'}\right)^2-4\,(N-1)\,\dfrac{p}{q},
\]
which is the same as
\begin{equation}
\label{maggiore1}
N-2-\dfrac{p}{q'}<(N-1)\,\dfrac{p}{q}\quad \Longleftrightarrow\quad N-2-p<(N-2)\,\frac{p}{q}.
\end{equation}
By recalling that we are in the case $p<N-1$ and we are assuming \eqref{condtotN1} and \eqref{delta}, we need to consider the two possibilities:
\begin{itemize}
\item[$\mathbf{A}$)] $p\le N-3$;
\vskip.2cm
\item[$\mathbf{B}$)] $N-3<p<\dfrac{(N-2)^2}{N-1}$.
\end{itemize}
In case $\mathbf{A})$, the second set of assumptions in \eqref{condtotN1} implies that \eqref{maggiore1} is verified and thus we are done. Observe that the bound
\[
q<\frac{(N-2)\,p}{N-2-p}.
\]
is compatible with $p\le N-3$ and \eqref{delta}, since for $p\le N-2$ we have
\[
\begin{split}
\frac{(N-2)\,p}{N-2-p}&>\dfrac{p}{\left(\sqrt{N-1}-\sqrt{p}\right)^2}
\quad\Longleftrightarrow\quad p\not=\frac{(N-2)^2}{N-1},
\end{split}
\]
and the latter is strictly greater than $N-3$.
\par
In case $\mathbf{B})$, in order to verify \eqref{maggiore1} we would need
\[
\left\{\begin{array}{rcl}
q&<&\dfrac{(N-2)\,p}{N-2-p},\\
&&\\
q&<&\dfrac{p}{p-(N-3)}.
\end{array}
\right.
\]
Observe that
\[
N-3< p< \frac{(N-2)^2}{N-1}\quad \Longrightarrow \quad \dfrac{(N-2)\,p}{N-2-p}< \dfrac{p}{p-(N-3)}.
\]
Thus the condition becomes
\[
\left\{\begin{array}{rcccl}
N-3&<&p&<&\dfrac{(N-2)^2}{N-1},\\
&&\\
&&q&<&\dfrac{(N-2)\,p}{N-2-p},
\end{array}
\right.
\]
which is again covered by our assumptions \eqref{condtotN1}. This concludes the proof.
\end{proof}

\section{Local Lipschitz estimate in dimension two}
\label{sec:5}

\subsection{Proof of Theorem \ref{teo:A}}

We now restrict the discussion to the case of dimension $N=2$ and consider the model case
\[
\mathfrak{F}(u;\Omega')=\sum_{i=1}^2 \frac{1}{p_i}\,\int_{\Omega'} (|u_{x_i}|-\delta_i)_+^{p_i}\,dx+\int_{\Omega'} f\,u\,dx,\qquad u\in W^{1,\mathbf{p}}_{{\rm loc}}(\Omega),\ \Omega'\Subset\Omega.
\]
We can suppose that $p_1<p_2$, since for $p_1=p_2$ the result has already proved in \cite{BBJ}.
Under the standing assumption, we take $U\in W^{1,\mathbf{p}}_{\rm loc}(\Omega)$ to be a local minimizer. Then we proceed as in \cite{BBJ}. 
\par
We take $\Omega'\Subset\Omega$ and set $d=\mathrm{dist}(\Omega',\partial \Omega)$. Since $\Omega'$ can be covered by a finite number of balls centered at $\Omega'$ and with radius $r_0\le d/100$, it is sufficient to show that
\[
\|\nabla U\|_{L^\infty(B_{r_0}(x_0))}<+\infty,
\]
where $B_{r_0}(x_0)$ is one of these balls.
To this aim, we set $B=B_{4\,r_0}(x_0)$ and solve the regularized problem for $0<\varepsilon\ll 1$
\begin{equation}
\label{regular}
\min\{\mathfrak{F}_\varepsilon(u;B)\, :\, u-U_\varepsilon\in W_0^{1,\mathbf{p}}(B)\},
\end{equation}
where:
\begin{itemize}
\item the regularized functional $\mathfrak{F}_\varepsilon$ is defined by
\[
\mathfrak{F}_\varepsilon(u;B)=\sum_{i=1}^2 \int_{B} g_{i,\varepsilon}(u_{x_i})\,dx+\int_{B} f_\varepsilon\,u\,dx;
\]
\item the functions $g_{i,\varepsilon}$ are given by
\[
g_{i,\varepsilon}(t)=\frac{(|t|-\delta_i)_+^{p_i}}{p_i}+\varepsilon\, \frac{t^2}{2},\qquad i=1,2;
\]
\item $U_\varepsilon$ and $f_\varepsilon$ are regularizations of $U$ and $f$.
\end{itemize}
By \cite[Theorem 2.4]{BBJ}, we know that \eqref{regular} admits a unique solution $u_\varepsilon$, which is smooth by proceeding as in \cite[Lemma 2.8]{BBJ}. In order to conclude, it is sufficient to prove the uniform estimate
\begin{equation}
\label{quasi0}
\|\nabla u_\varepsilon\|_{L^\infty(B_{r_0})}\le C,
\end{equation}
with $C>0$ independent of $\varepsilon$ and depending only on $p_1,p_2$, $\delta_1,\delta_2$, $r_0$, $\|f\|_{W^{1,\mathbf{p}'}(2\,B)}$ and $\|U\|_{W^{1,\mathbf{p}}(2\,B)}$. This is proved in the next subsection.
As in \cite{BBJ} (to which we refer for the missing details), this gives the estimate on $\nabla U$ and thus the conclusion. 

\subsection{Uniform Lipschitz estimate}

The proof of \eqref{quasi0} is the same as that of \cite[Proposition 4.1]{BBJ}, up to a couple of crucial modifications needed. We give the details of the latter and sketch the rest of the proof, by referring the reader to \cite{BBJ}.
For notational simplicity, we write $u$ in place of $u_\varepsilon$. 
We introduce the quantity
\[
\delta=1+\max\{\delta_1,\, \delta_2\}.
\]
then in what follows we set
\[
\mathcal{W}_i=\delta^2+(|u_{x_i}|-\delta)^2_+,\qquad i=1,2.
\]
First of all, we need the following Caccioppoli-type inequality. The proof is a slight variation of \cite[Lemma 3.6 \& Corollary 3.7]{BBJ}, we omit it.
\begin{lm}
There exists a constant $C=C(p_1,p_2)>0$ such that for every $s\ge 0$, every Lipschitz function $\eta$ with compact support in $B$ and $j=1,2$, we have
\begin{equation}
\label{basegg}
\begin{split}
\int \left|\Big(\mathcal{W}_j^{\frac{p_j}{4}+\frac{s}{2}}\Big)_{x_j}\right|^2\eta^2\, dx
&\le C\delta^{p_j-2}\left[\sum_{i=1}^2 \int \mathcal{W}_i^\frac{p_i-2}{2}\, \mathcal{W}_j^{s+1}|\nabla\eta|^2\,dx+(s+1)^2\int |f_\varepsilon|^2\,\mathcal{W}_j^{s}\, \eta^2\,dx\right].
\end{split}
\end{equation}
\end{lm}
We can now start the proof of the estimate \eqref{quasi0} for the gradient of $u$. We
may consider the case of the first component $u_{x_1}$ only, the other one being similar. With standard manipulations, from \eqref{basegg} we get 
\begin{equation}
\label{x1}
\begin{split}
\int \left|\left(\mathcal{W}_1^{\frac{p_1}{4}+\frac{s}{2}}\,\eta\right)_{x_1}\right|^2\, dx&\le C\,\delta^{p_1-2}\,\sum_{i=1}^2\int \mathcal{W}_i^\frac{p_i-2}{2}\, \mathcal{W}_1^{s+1}\,|\nabla\eta|^2\,dx\\
&+C\,\delta^{p_1-2}\,(s+1)^2\, \int |f_\varepsilon|^2\, \mathcal{W}_1^{s}\, \eta^2\, dx,
\end{split}
\end{equation}
with $C=C(p_1,p_2)>0$, where we used that $\delta \ge 1$.
In order to reconstruct the full gradient $\nabla \mathcal{W}_1^{\frac{p_1}{4}+\frac{s}{2}}$ on the left-hand side, we observe that 
\[
\left|\left(\mathcal{W}_1^{\frac{p_1}{4}+\frac{s}{2}}\right)_{x_2}\right|=\frac{p_1+2\,s}{p_1}\,\left|\left( \mathcal{W}_1^{\frac{p_1}{4}}\right)_{x_2}\right|\, \mathcal{W}_1^\frac{s}{2}.
\]
Then if we fix $1<q<2$, by H\"older's inequality with exponents $2/q$ and $2/(2-q)$,
we have
\[
\begin{split}
\left(\int \left|\left(\mathcal{W}_1^{\frac{p_1}{4}+\frac{s}{2}}\right)_{x_2}\right|^q\, \eta^q\, dx\right)^\frac{2}{q}&\le\left(\frac{p_1+2\,s}{p_1}\right)^2\left(\int \left|\left(\mathcal{W}_1^{\frac{p_1}{4}}\right)_{x_2}\right|^2\eta^2\, dx\right) \left(\int_{\mathrm{spt}(\eta)} \mathcal{W}_1^{\frac{q}{2-q}\,s}\, dx\right)^\frac{2-q}{q}\!\!.\\
\end{split}
\]
By using the same manipulations as in \cite{BBJ}, we thus get
\begin{equation}
\label{x2}
\begin{split}
\left(\int \left|\left(\mathcal{W}_1^{\frac{p_1}{4}+\frac{s}{2}}\,\eta\right)_{x_2}\right|^q\, dx\right)^\frac{2}{q}&\le C\,(1+s)^2\,\left(\int \left|\left(\mathcal{W}_1^{\frac{p_1}{4}}\right)_{x_2}\right|^2\eta^2\, dx\right)\,\left(\int_{\mathrm{spt}(\eta)} \mathcal{W}_1^{\frac{q}{2-q}\,s}\, dx\right)^\frac{2-q}{q}\\
&+C\,\left(\int \mathcal{W}_1^{\frac{p_1\,q}{4}+\frac{s\,q}{2}}\, |\eta_{x_2}|^q\, dx\right)^\frac{2}{q},
\end{split}
\end{equation}
with $C=C(p_1,p_2)>0$. We assume for simplicity that all the balls are centered at the origin.
We then fix the radius $r_0>0$ as above and define
\[
R_0=2\,r_0\qquad R_1:=\frac{3}{2}\,r_0.
\]
For $r_0<r<R<R_1$, we take $\eta\in W^{1,\infty}_0(B_R)$ to be the standard cut-off function
\[
\eta(x)=\min\left\{1,\frac{(R-|x|)_+}{R-r}\right\}.
\]
By multiplying \eqref{x1} and \eqref{x2} we get
\[
\begin{split}
\left(\int \left|\left(\mathcal{W}_1^{\frac{p_1}{4}+\frac{s}{2}}\,\eta\right)_{x_1}\right|^2\, dx\right)&\left(\int \left|\left(\mathcal{W}_1^{\frac{p_1}{4}+\frac{s}{2}}\,\eta\right)_{x_2}\right|^q\, dx\right)^\frac{2}{q}\\
&\le C\,\delta^{p_1-2}\,\left[\frac{1}{(R-r)^2}\,\sum_{i=1}^2 \int_{B_R} \mathcal{W}_i^\frac{p_i-2}{2}\, \mathcal{W}_1^{s+1}\, dx+(s+1)^2\, \int_{B_R} |f_\varepsilon|^2\, \mathcal{W}_1^{s}\, dx\right]\\
&\times \left[(s+1)^2\,\left(\int_{B_R} \left|\left(\mathcal{W}_1^{\frac{p_1}{4}}\right)_{x_2}\right|^2\, dx\right)\,\left(\int_{B_R} \mathcal{W}_1^{\frac{q}{2-q}\,s}\, dx\right)^\frac{2-q}{q}\right.\\
&\left.+\frac{1}{(R-r)^2}\,\left(\int_{B_R} \mathcal{W}_1^{\frac{p_1\,q}{4}+\frac{s\,q}{2}}\, dx\right)^\frac{2}{q}\right].
\end{split}
\]
Then we apply the anisotropic Sobolev inequality of Theorem \ref{teo:ST} to the compactly supported function $\mathcal{W}_1^{(p_1+2\,s)/4}\,\eta$. This yields
\begin{equation}
\label{gauche}
\begin{split}
\mathcal{T}_{q}&\bigg(\int\left(\mathcal{W}_1^{\frac{p_1}{4}+\frac{s}{2}}\eta\right)^{\overline q^*} dx\bigg)^\frac{4}{\overline q^*}\\
&\le C\,\delta^{p_1-2}\,\bigg[\frac{1}{(R-r)^2}\sum_{i=1}^2 \int_{B_R} \mathcal{W}_i^\frac{p_i-2}{2} \mathcal{W}_1^{s+1} dx +(s+1)^2 \int_{B_R} |f_\varepsilon|^2\, \mathcal{W}_1^{s}\, dx\bigg]\\
&\times \bigg[(s+1)^2\bigg(\int_{B_R} \bigg|\left(\mathcal{W}_1^{\frac{p_1}{4}}\right)_{x_2}\bigg|^2\, dx\bigg)\left(\int_{B_R} \mathcal{W}_1^{\frac{q}{2-q}\,s}\, dx\right)^\frac{2-q}{q}\\
&+\frac{1}{(R-r)^2}\,\left(\int_{B_R} \mathcal{W}_1^{\frac{p_1\,q}{4}+\frac{s\,q}{2}}\, dx\right)^\frac{2}{q}\bigg].
\end{split}
\end{equation}
The exponents $\overline q$ and $\overline q^*$ are given by
\[
\overline{q}=\frac{4\,q}{2+q}\qquad \mbox{ and }\qquad \overline q^*=\frac{4\,q}{2-q},
\]
the constant $\mathcal{T}_q$ only depends on $q$ and it degenerates to $0$ as $q$ goes to $2$.
\par
The idea is to use the previous fundamental estimate \eqref{gauche} to produce an iterative scheme of reverse H\"older inequalities on shrinking balls. Then we perform a Moser's iteration in order to conclude.
We need to estimate the terms appearing in the right-hand side of \eqref{gauche}. The crucial difference with respect to \cite{BBJ} is in the first term on the right-hand side of \eqref{gauche}, i.e. 
\begin{equation}\label{right}
\sum_{i=1}^2 \int_{B_R} \mathcal{W}_i^\frac{p_i-2}{2}\, \mathcal{W}_1^{s+1}\, dx=\int_{B_R} \mathcal{W}_1^\frac{p_1}{2}\, \mathcal{W}_1^{s}\, dx+\int_{B_R} \mathcal{W}_2^\frac{p_2-2}{2}\, \mathcal{W}_1\,\mathcal{W}_1^{s}\, dx.
\end{equation}
On the contrary, all the other terms are estimated exactly as in \cite{BBJ}, thus we omit the details. Let us now focus on the term above, it is useful to introduce the quantity
\[
\begin{split}
\widetilde{\mathcal{I}}(\mathcal{W}_1,\mathcal{W}_2,f_\varepsilon;R_0,R_1)&=\sum_{i=1}^2 \left[\left(\frac{R_0}{R_1}\right)^2\,\fint_{B_{R_0}} \mathcal{W}_i^\frac{p_i}{2}\, dx +\int_{B_{R_1}} \left|\nabla \mathcal{W}_i^\frac{p_i}{4}\right|^2\, dx\right]^{\frac{p_i-2}{p_i}\,\frac{p_1}{p_1-2}}\\
&+R_0^{\frac{2}{p_1}}\, \left(\int_{B_{R_1}} |f_\varepsilon|^{2\,p'_1}\, dx\right)^\frac{1}{p'_1}.
\end{split}
\]
First we claim that $\widetilde{\mathcal{I}}(\mathcal{W}_1,\mathcal{W}_2,f_\varepsilon;R_0,R_1)$ is uniformly bounded, independently of $\varepsilon$. To this aim, as for the term containing $f_\varepsilon$, we observe that by Proposition \ref{prop:KK} we have the continuous embedding (recall that $R_1<R_0$)
\[
W^{1,\mathbf{p}'}(B_{R_0})\hookrightarrow L^{2\,p_1'}(B_{R_1}),\qquad \mbox{ since } p_2'<p_1'\le 2 \ \mbox{ and }\ 2\,p_1'<\overline{p'}^*=\frac{2\,p_1'\,p_2'}{p_1'+p_2'-p_1'\,p_2'},
\] 
thus the term
\[
\left(\int_{B_{R_1}} |f_\varepsilon|^{2\,p'_1}\, dx\right)^\frac{1}{p'_1},
\]
can be uniformly bounded in terms of the $W^{1,\mathbf{p}'}$ norm of $f$ on $B_{R_0}$. The terms containing the gradients of $\mathcal{W}_1^{p_1/4}$ and $\mathcal{W}_2^{p_2/4}$ are more delicate, for them we need 
Theorem \ref{teo:sobolevpq}. Indeed, let us define 
\[
V_{i,\varepsilon}(t)=\int_0^t \sqrt{g_{i,\varepsilon}''(s)}\,ds\qquad \mbox{ and }\qquad \mathcal{V}_{i,\varepsilon}=V_{i,\varepsilon}((u_\varepsilon)_{x_i}),\quad i=1,2.
\]
We observe that $V_{i,\varepsilon}:\mathbb{R}\to\mathbb{R}$ is a locally Lipschitz omeomorphism, with $V_{i,\varepsilon}'>0$. 
If we set 
\[
\mathfrak{f}_i(t)=\left(\delta^2+(|t|-\delta)^2_+\right)^\frac{p_i}{4},\qquad t\in\mathbb{R},
\]
then we obtain that $\mathcal{W}_i^{p_i/4}=\Phi_{i,\varepsilon}(\mathcal{V}_{i,\varepsilon})$, where 
\[
\Phi_{i,\varepsilon}(t)=\mathfrak{f}_i(V_{i,\varepsilon}^{-1}(t)),\qquad t\in\mathbb{R}.
\]
It is not difficult to see that $\Phi_{i,\varepsilon}$ is a Lipschitz function, with Lipschitz constant independent of $\varepsilon$. Indeed, we have 
\[
\mathfrak{f}'_{i}(t)=0,\quad \mbox{ for } |t|<\delta\qquad \mbox{ and }\qquad |\mathfrak{f}'_i(t)|\le \sqrt{C_i}\,|t|^\frac{p_i-2}{2},\quad \mbox{ for } |t|\ge \delta.
\]
\[
V'_{i,\varepsilon}(t)=\sqrt{g_{i,\varepsilon}''(t)}\ge \frac{1}{\sqrt{C_i}}\,|t|^\frac{p_i-2}{2},\quad \mbox{ for } |t|\ge \delta,
\]
for some $C_i=C_i(p_i,\delta)\ge 1$. Thus we get
\[
|\Phi_{i,\varepsilon}'(t)|=\left|\mathfrak{f}_i'(V_{i,\varepsilon}^{-1}(t))\right|\,\frac{1}{V'_{i,\varepsilon}(V_{i,\varepsilon}^{-1}(t))}\le C_i,\qquad t\in\mathbb{R}.
\]
By using this observation, we thus obtain
\[
\int_{B_{R_1}} \left|\nabla \mathcal{W}_i^\frac{p_i}{4}\right|^2\,dx\le L_i\, \int_{B_{R_1}} |\nabla \mathcal{V}_{i,\varepsilon}|^2\,dx,
\]
with $L_i=L_i(p_i,\delta)>0$. We can now invoke \eqref{stimammerda}, in order to bound uniformly the last term. It is only left to observe that the bound in \eqref{stimammerda} also depends on the local $L^\infty$ norm of $u_\varepsilon$. This can be uniformly bounded by appealing to \cite[Theorem 3.1]{FS}, proving the claim.

We now come back to estimate the quantities in \eqref{right}. Let us recall that, since we are in dimension $N=2$, we have the continuous embedding $W^{1,2}(B_{R_1})\hookrightarrow L^\vartheta(B_{R_1})$ for every $1\le \vartheta<+\infty$. Then by H\"older's inequality and Sobolev-Poincar\'e inequality, exactly as in \cite{BBJ} we get
\[
\begin{split}
\int_{B_R} \mathcal{W}_1^\frac{p_1}{2}\,\mathcal{W}_1^{s}\, dx
&\le C\, \widetilde{\mathcal{I}}(\mathcal{W}_1,\mathcal{W}_2,f_\varepsilon;R_0,R_1)\, R_0^{\frac{2}{p'_1}} \left(\int_{B_R} \mathcal{W}_1^{s\,p_1}\, dx\right)^\frac{1}{p_1}.
\end{split}
\]
For the second term we have to be more careful. By using H\"older inequality with exponents
\[
p_1'\,\frac{p_1}{p_1-2},\qquad p_1'\,\frac{p_1}{2},\qquad  p_1,
\]
we get
\[
\begin{split}
\int_{B_R} \mathcal{W}_2^\frac{p_2-2}{2}\,\mathcal{W}_1\,\mathcal{W}_1^{s}\, dx&\le C\,\left[\left(\int_{B_{R_1}} \left(\mathcal{W}_2^\frac{p_2}{4}\right)^{2\,p'_1\,\frac{p_2-2}{p_1-2}\,\frac{p_1}{p_2}}\, dx\right)^{\frac{1}{p_1'}}+\left(\int_{B_{R_1}}\left(\mathcal{W}_1^\frac{p_1}{4}\right)^{2\,p'_1}\, dx\right)^\frac{1}{p_1ì}\right]\\
&\times\left(\int_{B_R} \mathcal{W}_1^{s\,p_1}\, dx\right)^\frac{1}{p_1},
\end{split}
\]
where we further used Young's inequality and the constant $C=C(p_1)>0$ depends only on $p_1$. To treat the term into square brakets, we use again Sobolev-Poincar\'e inequalities. Namely, we have
\[
\left(\int_{B_{R_1}} \left(\mathcal{W}_1^\frac{p_1}{4}\right)^{2\,p'_1}\, dx\right)^\frac{1}{p'_1}\le C\, R^\frac{2}{p_1'}_1\,\left[\fint_{B_{R_1}} \mathcal{W}_1^\frac{p_1}{2}\,dx+\int_{B_{R_1}} \left|\nabla \mathcal{W}_1^\frac{p_1}{4}\right|^2\,dx\right],
\]
and
\[
\left(\int_{B_{R_1}} \left(\mathcal{W}_2^\frac{p_2}{4}\right)^{2\,p'_1\,\frac{p_2-2}{p_1-2}\,\frac{p_1}{p_2}}\, dx\right)^{\frac{1}{p_1'}\,\frac{p_1-2}{p_2-2}\,\frac{p_2}{p_1}}\le C\, R_1^{\frac{2}{p_1'}\,\frac{p_1-2}{p_2-2}\,\frac{p_2}{p_1}}\,\left[\fint_{B_{R_1}} \mathcal{W}_2^\frac{p_2}{2}\,dx+\int_{B_{R_1}}\left|\nabla \mathcal{W}_2^\frac{p_2}{4}\right|^2\,dx\right].
\]
Thus we obtain
\[
\begin{split}
\int_{B_R} \mathcal{W}_2^\frac{p_2-2}{2}\,\mathcal{W}_1\,\mathcal{W}_1^{s}\, dx&\le C\, \widetilde{\mathcal{I}}(\mathcal{W}_1,\mathcal{W}_2,f_\varepsilon;R_0,R_1)\, R_0^{\frac{2}{p'_1}} \left(\int_{B_R} \mathcal{W}_1^{s\,p_1}\, dx\right)^\frac{1}{p_1},
\end{split}
\]
as well, where we used again that $R_1<R_0$.
\vskip.2cm\noindent
By using these estimates in \eqref{gauche} and proceeding as in \cite{BBJ} for all the other terms, we obtain
\begin{equation}
\label{reverseg}
\begin{split}
\left[\int_{B_r} \left(\mathcal{W}_1^{\frac{p_1}{2}+s}\right)^\frac{2\,q}{2-q}\, dx\right]^\frac{2-q}{q}&\le C\,\delta^{p_1-2}\,\left[\left(\frac{R_0}{R-r}\right)^2\,\widetilde{\mathcal{I}}(\mathcal{W}_1,\mathcal{W}_2,f_\varepsilon;R_0,R_1)\,R_0^{-\frac{2}{p_1}} \left(\int_{B_R} \mathcal{W}_1^{s\,p_1}\, dx\right)^\frac{1}{p_1}\right.\\
&\left.+(s+1)^2\,\widetilde{\mathcal{I}}(\mathcal{W}_1,\mathcal{W}_2,f_\varepsilon;R_0,R_1)\,R_0^{-\frac{2}{p_1}}\,\left(\int_{B_{R}} \mathcal{W}_1^{s\,p_1}\, dx\right)^\frac{1}{p_1}\right]\\
&\times \left[(s+1)^2\,\widetilde{\mathcal{I}}(\mathcal{W}_1,\mathcal{W}_2,f_\varepsilon;R_0,R_1)\left(\int_{B_R} \mathcal{W}_1^{\frac{q}{2-q}\,s}\, dx\right)^\frac{2-q}{q}\right.\\
&\left.+\left(\frac{R_0}{R-r}\right)^2\,R_0^{2\,\left(\frac{2}{q}-\frac{1}{p_1}-1\right)}\,\widetilde{\mathcal{I}}(\mathcal{W}_1,\mathcal{W}_2,f_\varepsilon;R_0,R_1)\, \left(\int_{B_R} \mathcal{W}_1^{s\,p_1}\, dx\right)^\frac{1}{p_1}\right],
\end{split}
\end{equation}
for a constant $C=C(p_1,q)>0$. The exponent $1<q<2$ is chosen as
\[
q= \frac{2\,p_1}{p_1+1},\qquad  \mbox{ so that }\ \frac{q}{2-q}=p_1\quad \mbox{ and }\quad \frac{2}{q}-\frac{1}{p_1}-1=0.
\]
By further observing that $\mathcal{W}_1\ge 1$, from \eqref{reverseg} we gain
\[
\begin{split}
\left(\int_{B_r} \mathcal{W}_1^{2\,s\,p_1}\, dx\right)^\frac{1}{p_1}&\le C\,\delta^{p_1-2}\,\widetilde{\mathcal{I}}(\mathcal{W}_1,\mathcal{W}_2,f_\varepsilon;R_0,R_1)^2\\
&\times\left[\left(\frac{R_0}{R-r}\right)^2+(s+1)^2\,\right]^2\,R_0^{-\frac{2}{p_1}}\,\left(\int_{B_R} \mathcal{W}_1^{s\,p_1}\, dx\right)^\frac{2}{p_1},
\end{split}
\]
for $s\ge 0$.
This is an iterative scheme of reverse H\"older inequalities, we can now iterate infinitely many times this estimate, as in \cite{BBJ}.

\appendix

\section{Pointwise inequalities}

\begin{lm}
Let $g:\mathbb{R}\to\mathbb{R}^+$ be a $C^{1,1}$ convex function. Let us set
\[
V(t)=\int_0^t \sqrt{g''(\tau)}\,d\tau.
\]
For every $a,b\in\mathbb{R}$ we have
\begin{equation}
\label{monotone}
\Big(g'(a)-g'(b)\Big)\,(a-b)\ge  \left|V(a)-V(b)\right|^2.
\end{equation}
\end{lm}
\begin{proof}
Without loss of generality, we can assume that $a\ge b$. Indeed, $g'(a)-g'(b)$ and $a-b$ have the same sign, thanks to the monotonicity of $g'$. For $a=b$ there is nothing to prove, so we take $a>b$. By using Jensen inequality, we have
\[
\begin{split}
\Big(g'(a)-g'(b)\Big)\,(a-b)&=\left(\int_b^a g''(t)\,dt\right)\,(a-b)\\
&\ge \left(\int_b^a \sqrt{g''(t)}\,dt\right)^2=\left(V(a)-V(b)\right)^2,
\end{split}
\]
as desired.
\end{proof}
\begin{lm}
Let $g:\mathbb{R}\to\mathbb{R}^+$ be a $C^{1,1}$ convex increasing function. For every $a,b\in\mathbb{R}$ we have
\begin{equation}
\label{lipschitz}
\left|g'(a)-g'(b)\right|\le  \sup_{s \in [a,b]} \Big(\sqrt{g''(s)}\Big)\,|V(a)-V(b)|.
\end{equation}
\end{lm}
\begin{proof}
For $\varepsilon>0$, let us consider the function $g_\varepsilon(t)=g(t)+\varepsilon\,t^2$. We set
\[
V_\varepsilon(t)=\int_0^t \sqrt{g_\varepsilon''(\tau)}\,d\tau,
\]
then we observe that this is a strictly increasing function, thus invertible. Finally, we define
\[
F_\varepsilon(t)=g'_\varepsilon\left(V_\varepsilon^{-1}(t)\right),
\]
which is an increasing function. Indeed, we have
\[
F'_\varepsilon(t)=g''_\varepsilon(V^{-1}_\varepsilon(t))\,\frac{1}{V'_\varepsilon(V_\varepsilon^{-1}(t))}=\sqrt{g_\varepsilon''(V^{-1}_\varepsilon(t))}>0.
\]
By basic Calculus, this yields
\[
\begin{split}
|g'_\varepsilon(a)-g'_\varepsilon(b)|=|F_\varepsilon(V_\varepsilon(a))-F_\varepsilon(V_\varepsilon(b))|&\le \sup_{s \in [a,b]}\Big(F'_\varepsilon(V_\varepsilon(s))\Big)\,|V_\varepsilon(a)-V_\varepsilon(b)|\\
&= \sup_{s \in [a,b]} \Big(\sqrt{g_\varepsilon''(s)}\Big)\,|V_\varepsilon(a)-V_\varepsilon(b)|.
\end{split}
\]
By taking the limit as $\varepsilon$ goes to $0$, we get the desired conclusion.
\end{proof}
\begin{oss}
When $g(t)=|t|^p/p$, the previous inequalities imply the familiar estimates
\[
\Big(|a|^{p-2}\, a-|b|^{p-2}\, b\Big)\,(a-b)\ge (p-1)\,\frac{4}{p^2}\, \left||a|^\frac{p-2}{2}\,a-|b|^\frac{p-2}{2}\, b\right|^2.
\]
and
\[
\Big||a|^{p-2}\,a-|b|^{p-2}\,b\Big|\le 2\,\frac{p-1}{p}\, \left(|a|^\frac{p-2}{2}+|b|^\frac{p-2}{2}\right)\,\left||a|^\frac{p-2}{2}\,a-|b|^\frac{p-2}{2}\,b\right|.
\]
\end{oss}

\end{document}